\documentclass[12pt]{amsart}
\usepackage{amscd,amsmath,amssymb,amsfonts}
\usepackage[cmtip, all]{xy}

\newtheorem{thm}{Theorem}[section]
\newtheorem{prop}[thm]{Proposition}
\newtheorem{lem}[thm]{Lemma}
\newtheorem{cor}[thm]{Corollary}

\theoremstyle{definition}

\newtheorem{defn}[thm]{Definition}
\newtheorem{remk}[thm]{Remark}
\newtheorem{remks}[thm]{Remarks}

\newtheorem{exm}[thm]{Example}
\newtheorem{exms}[thm]{Examples}
\newtheorem{notat}[thm]{Notation}
\numberwithin{equation}{section}

{\hfill$\square$\end{defn}}

{\hfill$\square$\end{remk}}

{\hfill$\square$\end{remks}}

{\hfill$\square$\end{exm}}

{\hfill$\square$\end{exms}}

{\hfill$\square$\end{notat}}

\newcommand{\sA}{{\mathcal A}}
\newcommand{\sB}{{\mathcal B}}
\newcommand{\sC}{{\mathcal C}}
\newcommand{\sD}{{\mathcal D}}
\newcommand{\sE}{{\mathcal E}}
\newcommand{\sF}{{\mathcal F}}
\newcommand{\sG}{{\mathcal G}}
\newcommand{\sH}{{\mathcal H}}
\newcommand{\sI}{{\mathcal I}}

\newcommand{\sK}{{\mathcal K}}

\newcommand{\sM}{{\mathcal M}}
\newcommand{\sN}{{\mathcal N}}
\newcommand{\sO}{{\mathcal O}}

\newcommand{\sT}{{\mathcal T}}
\newcommand{\sU}{{\mathcal U}}

\newcommand{\C}{{\mathbb C}}

\renewcommand{\H}{{\mathbb H}}

\newcommand{\R}{{\mathbb R}}

\newcommand{\Z}{{\mathbb Z}}

\newcommand{\surj}{\twoheadrightarrow}
\newcommand{\inj}{\hookrightarrow}

\newcommand{\wt}{\widetilde}

\newcommand{\ds}{{/\kern-3pt/}}

\newcommand{\un}{\underline}
\newcommand{\ov}{\overline}

\begin{document}
\title{Nisnevich descent for Deligne Mumford stacks}
\author{Amalendu Krishna, Paul Arne {\O}stvaer}
\address{School of Mathematics, Tata Institute of Fundamental Research,  
Homi Bhabha Road, Colaba, Mumbai, India.}
\email{amal@math.tifr.res.in}
\address{Department of Mathematics, University of Oslo,
PO Box 1053, Blindern, Oslo, Norway.}
\email{paularne@math.uio.no} 
\baselineskip=10pt 
  

\maketitle

\begin{abstract}
We prove the excision property for the $K$-theory of perfect complexes on 
Deligne-Mumford stacks. We define the Nisnevich site on the category
of such stacks which restricts to the usual Nisnevich site on schemes.
Using the refinements of the localization sequence of \cite{Krishna} and
\cite{Toen}, we then show that the $K$-theory of perfect complexes satisfies 
the Nisnevich descent on the category of tame Deligne-Mumford stacks.
This is done by showing that the above Nisnevich site is given by
a $cd$-structure which is complete, regular and bounded. 
\end{abstract}
\section{Introduction}
It is by now well known that the localization and the Mayer-Vietoris
are one of the fundamental desired properties for the algebraic $K$-theory of 
schemes. It was shown by Brown and Gersten \cite{BG} that the above two 
properties yield the Zariski descent for the $K$-theory of 
coherent sheaves on schemes. In this paper, we concentrate on the study of
such properties for the $K$-theory of perfect complexes which is
a much more difficult task than the $K$-theory of coherent sheaves.
 
On of the main features of the descent property is that it essentially reduces
the problem of computing the $K$-group of a scheme in terms of the
$K$-groups of local schemes which is expectedly an easier task. For the
reason of reducing the study of $K$-theory to the local calculations on
better behaved schemes and stacks, it is only desirable that one has a
descent property for the $K$-theory with respect to the Zariski and possibly
finer topologies. The Zariski and the Nisnevich descent for the 
$K$-theory of perfect complexes on schemes was established by Thomason
and Trobaugh \cite{TT}. They proved it by showing that this $K$-theory
has the localization and the excision properties, both of which are difficult
results in their own way. 

The K-theory of quotient stacks that result from the action of a given group 
on schemes was developed by Thomason \cite{Thomason2}. The algebraic 
$K$-theory of stacks were later studied in \cite{Joshua1} and \cite{Toen1},
where the localization sequence for the $K$-theory of coherent sheaves
was established. The general form of the localization sequence for the 
$K$-theory of perfect complexes on stacks is still not known. However,
for the Deligne-Mumford stacks which are tame, have coarse moduli schemes
and have the resolution property, such a localization sequence was first 
shown to exist in \cite{Krishna}. A more general and categorical form of such a
localization sequence has been recently obtained by T{\"o}en \cite{Toen}.

The first form of descent theorem for the $K$-theory of stacks was 
proven by Thomason in \cite{Thomason3}, where he showed that the
mod-$n$ $K$-theory of perfect complexes on smooth quotient stacks arising 
from the action of a fixed smooth and affine group scheme on schemes, satisfies
the descent with respect to the isovariant \'etale topology if one
inverts the Bott element. Our goal in this paper is to prove the 
localization, excision and descent theorems for the $K$-theory of perfect 
complexes on Deligne-Mumford stacks in as general form as possible. We
briefly describe these results in the following paragraphs.

Let $S$ be a fixed noetherian affine scheme. We denote the category of
separated Deligne-Mumford stacks of finite type over $S$ by ${\sD}{\sM}_S$. 
Using the characterization of perfect complexes in terms of the compact
objects in the unbounded derived category of complexes of sheaves of
$\sO_X$-modules with quasi-coherent cohomology sheaves and its version with 
support, and then combining this with the results of \cite{Toen}, we first
establish the appropriate localization for the $K$-theory of perfect 
complexes on separated and tame Deligne-Mumford stacks in 
Theorem~\ref{thm:Localization}. Our next main result of this paper
is Theorem~\ref{thm:Excision}, where we establish an excision theorem
for the above class of stacks. This is the first result of its kind for
stacks and generalizes the corresponding result of \cite{TT} in the
world of stacks. As a consequence of the above results, we show in 
Corollary~\ref{cor:MV} that the $K$-theory of perfect complexes
satisfies the Mayer-Vietoris property in the category of separated and tame
Deligne-Mumford stacks with coarse moduli schemes. This plays a very crucial
role in our proof of the Nisnevich descent for the $K$-theory.

In Section~\ref{section:NIS}, we introduce the Nisnevich site on the 
category ${\sD}{\sM}_S$. It turns out that this is a Grothendieck site 
and its restriction to the full subcategory of schemes is the
usual Nisnevich site for schemes. 
We also introduce a Nisnevich $cd$-structure on ${\sD}{\sM}_S$ in the
sense of \cite{Voev1}. In Theorem~\ref{thm:Nis-cd-main}, we show that
the Nisnevich site of ${\sD}{\sM}_S$ is in fact given by the above
$cd$-structure. It is further shown in Section~\ref{section:Nis-CRB}
that this $cd$-structure is complete, regular and bounded.
Our final main result is Theorem~\ref{thm:MNDS}, where we show that
the $K$-theory of perfect complexes satisfies the Nisnevich descent for
representable maps. We also draw several consequences, one of which is
to show that the Nisnevich (and Zariski) cohomological dimension of a stack 
in ${\sD}{\sM}_S$ is bounded by the dimension of the stack. This is first
result of its kind for stack. We also deduce a Brown-Gertsten spectral
sequence for the $K$-theory of such stacks.

The computation of the equivariant algebraic $K$-theory and 
equivariant higher Chow groups of schemes with group action has been
a very active area of research ever since Thomason developed this theory.
It often turns out the the equivariant K-theory is easier to compute
using the tools of representation theory and one can often use them to
compute the non-equivariant $K$-theory via spectral sequences. 
We refer to \cite{Krishna1} for some results in this direction and
to \cite{Merkurjev} for applications of these techniques. It is often
difficult to prove various expected results about the equivariant $K$-theory
which are otherwise known in the non-equivariant case. Our 
perspective is to look at these equivariant $K$-groups as special cases
of the $K$-groups of stacks, which constitute a much larger class of
geometric objects and where one can develop the general theory
and prove those results for stacks which are hitherto known only for schemes,
using the categorical point of view. 

We end the introduction with some remarks about our ongoing and future
work of which the present paper is the first part. Our main motivation
for writing this paper was to undertake a deeper study of the $K$-theory
of perfect complexes on stacks. In the upcoming sequel \cite{CKO} to 
this paper, we study the homotopy $K$-theory of Deligne-Mumford stacks, 
introduce the $cdh$-site of such stacks and the final goal there is to
show that the homotopy $K$-theory of stacks satisfies the $cdh$-descent.
We hope to deduce several important consequences of these results for the
$K$-theory of stacks which have been proven for schemes in recent years.

Finally, as the reader will observe, our focus in this paper is to study 
certain fundamental properties of the $K$-theory of Deligne-Mumford stacks. 
There is a vast majority of stacks occurring in geometry which are
of this kind. A very interesting class of such stacks are the moduli stacks 
of stable maps from $n$-pointed stable curves of genus $g$ to a projective 
variety $X$, denoted by ${\sM}_{g,n}(X, \beta)$. There has been a lot
of interset in the $K$-theory and quantum $K$-theory of these moduli
stacks from the point of view of Gromov-Witten theory.
One should notice that these moduli stacks are mostly singular even if $X$ is 
smooth. Nonetheless, it would be interesting to know if the results of this 
paper can be proven for more general Artin stacks.
As the techniques used in this paper suggest, such a generalization
is very much expected and we hope to take this up in a future project.
 
\section{Preliminaries}\label{section:Prelims}
In this section, we recall the definition and some basic properties of 
Deligne-Mumford stacks which we shall use frequently. We refer the reader to 
\cite{LMB} for more details. All the stacks in this article will be
defined over a fixed noetherian and base scheme $S$. In some cases, we shall 
take $S$ to be the spectrum of a field in this paper.  

\begin{defn}\label{defn:stacks} A stack $X$ is called a Deligne-Mumford
stack if \\
$(1)$ The diagonal ${\Delta}_X : X \to X {\times}_S X$ is representable,
quasi-compact and separated. \\
$(2)$ There is an $S$-scheme $U$ and an \'etale surjective morphism
$U \to X$.
\end{defn}
The scheme $U$ is called an atlas of $X$. Note that the representability
of the diagonal implies that any atlas $U \to X$ such as in the second
condition above is automatically representable.
The stack $X$ is called separated if the diagonal is proper. It is
called a noetherian stack if $U$ can be chosen to be a noetherian
scheme. An algebraic space over $S$ is a Deligne-Mumford stack where
the diagonal is an embedding. It is separated if the diagonal is
a closed embedding. We say that $X$ is of finite type over $S$ if it
has an atlas which is of finite type over $S$. The following well known
theorem provides a large class of examples of Deligne-Mumford stacks.

\begin{thm}[Deligne-Mumford, \cite{DM}]\label{thm:DM}
Let $X/S$ be a noetherian scheme of finite type and let $G/S$ be a
smooth affine group scheme (of finite type over $S$) acting on $X$
such that the stabilizers of geometric points are finite and reduced.
Then the quotient stack $[X/G]$ ({\sl cf.} \cite[7.17]{Vistoli}) is a 
noetherian
Deligne-Mumford stack. If the stabilizers are trivial, then $[X/G]$ is 
an algebraic space. Furthermore, the stack is separated if and only if the 
action is proper.
\end{thm} 

\begin{remk}\label{remk:general}
If the action ${\Psi}: G \times X \to X \times X$ is proper, then
the fiber of the action map over a geometric point $x$ of $X$ is
a closed group subscheme of $G$ which is proper over ${\rm Spec}(k(x))$.   
Now since $G$ is affine over $S$, this fiber must be finite. Thus 
proper actions have finite stabilizers. In particular, if $S$ is
defined over a field of characteristic zero, then the quotient stack
$[X/G]$ for a proper action is a Deligne-Mumford stack. However,
in positive characteristic, this need not be true since the stabilizers
of geometric points can be non-reduced.
\end{remk} 

A {\sl stack} in this paper will always mean a noetherian and separated
Deligne-Mumford stack over a noetherian base scheme $S$,
unless mentioned otherwise. For a noetherian and
affine base scheme $S$, we shall denote the category of noetherian and
separated Deligne-Mumford stacks over $S$ by ${\sD}{\sM}_S$. Let $Aff/S$ 
denote the category of morphisms $T \to S$, where $T$ is an affine scheme.
The category ${\sD}{\sM}_S$ is clearly closed under fiber products. To say 
more about this category, recall that a stack $X$ is quasi-compact if it 
has an \'etale atlas $U \to X$ such that $U$ is quasi-compact. In particular, 
all noetherian stacks are quasi-compact. We also recall that a morphism 
$f : X \to Y$ of stacks is {\sl quasi-compact} if for all $T \in Aff/S$ and
all morphisms $T \xrightarrow{g} Y$, the stack $X {\times}_Y T$ is 
quasi-compact. A morphism $f : X \to Y$ of stacks is called {\sl separated} if
the diagonal $\Delta_X : X \to X {\times}_Y X$ is proper, i.e., it is
universally closed. The morphism $f$ is said to be {\sl representable} if
for every morphism $Z \xrightarrow{g} Y$, where $Z$ is an algebraic space,
the stack $X {\times}_Y Z$ is also an algebraic space. We say that $f$ is
{\sl strongly representable} if $X \times_Y Z$ is a scheme whenever
$Z$ is a scheme. It follows from \cite[Lemme~4.2]{LMB} that any morphism
$f : X \to Y$, where $X$ is a scheme, is strongly representable. This
fact will be used in this paper repeatedly without any further mention
of the above reference. 

\begin{prop}\label{prop:QCS}
Every morphism $f : X \to Y$ in ${\sD}{\sM}_S$ is quasi-compact and
separated.
\end{prop}
\begin{proof}
We first show that $f$ is quasi-compact. 
Let $U \xrightarrow{u} Y$ be a noetherian affie atlas and let 
$f': Z = X {\times}_Y U \to U$ be the base change. We first claim that the 
quasi-compactness of $f'$ implies the same for $f$. To prove this claim,
let $T \xrightarrow{g} Y$ be a morphism with $T \in
Aff/S$ and let $X' = X {\times}_Y T$. Consider the following commutative
diagram, where $W = U {\times}_Y T$.
\[
\xymatrix@C.5pc{
V \ar[r] \ar[ddr] & T {\times}_Y Z  \ar[rr] \ar[dd] \ar[dr] & 
& X'\ar[dr] \ar[dd] & \\
&  & Z \ar[rr] \ar[dd]_{f'} & & X \ar[dd]^{f} \\
& W \ar[rr] \ar[dr] & & T \ar[dr] & \\
& & U \ar[rr]_{u} & & Y}
\]
Since $Y$ is separated and $U$, $T$ are affine, we see that $W$ is affine.
Now the quasi-compactness of $f'$ implies that $T {\times}_Y Z$ is 
quasi-compact and hence there is an atlas $V \to T {\times}_Y Z$ which
is quasi-compact. On the other hand, as $u$ is \'etale and surjective,
we see that $T {\times}_Y Z \to X'$ is also \'etale and surjective. 
In particular, $V \to X'$ is an atlas, which shows that $X'$ is quasi-compact.
This proves the claim. 

Using the claim, we can assume that $Y$ is a noetherian affine scheme. Let
$T \to U$ be a morphism with $T \in Aff/S$ and put $X' = T {\times}_Y X$. 
Let $V \xrightarrow{v} X$ be a noetherian atlas of $X$. Then it is easy to
see that $V {\times}_Y T = W \to X'$ is an \'etale atlas of $X'$. On the
other hand, $V \to Y$ is a morphism of noetherian schemes and hence 
quasi-compact, which means that $W$ is a quasi-compact
scheme. In particular, $f$ is quasi-compact.

We now prove the separatedness of $f$. Since $X$ and $Y$ are separated
Deligne-Mumford stacks over $S$, it suffices to show that if 
$X \xrightarrow{f} Y \xrightarrow{g} Z$ are morphisms of stacks such that 
$h = g \circ f$ is separated, then $f$ is also separated. This is easily proved
using the valuative criteria of separatedness. So let $V = {\rm Spec}(V)$
be the spectrum of a valuation ring with the valuation field $K$. Let
$U = {\rm Spec}(K)$ and let $t : V \to Y$. Put $t' = g \circ t$. Let
$x_1, x_2 \in Ob(X_V)$ and let $\beta : f \circ x_1 \xrightarrow{\cong} f
\circ x_2$ and $\alpha : (x_1)_U \xrightarrow{\cong} (x_2)_U$ be given
such that $f(\alpha) = (\beta)_U$. 
Put ${\beta}' = g(\beta) : h \circ x_1 \xrightarrow{\cong} h \circ x_2$.
This implies that $h(\alpha) = ({\beta}')_U$.
Hence the separatedness of $h$ implies that
there exists $\wt{\alpha} : x_1 \xrightarrow{\cong} x_2$ extending 
$\alpha$. In particular, $f$ is separated.
\end{proof}       

Let ${\bf P}$ be a property of morphisms of schemes which is of
local nature for the \'etale topology.
Examples of such properties are a morphism being flat, smooth, \'etale,
unramified, etc.

\begin{defn}\label{defn:morphism}
A morphism $f : X \to Y$ of stacks is said to have property 
${\bf P}$ if there are \'etale atlases $U \to X$, $V \to Y$ and 
a compatible morphism $U \to V$ with the property ${\bf P}$.
\end{defn}

\begin{defn}\label{defn:morphism1}
Let ${\bf P}$ be a property of schemes which is of local nature
in the \'etale topology.  
We say that the stack $X$ has property ${\bf P}$ if there
is an (\'etale) atlas $U \to X$ such that $U$ has property ${\bf P}$.
The stack $X$ is said to have dimension $n$ if it has an atlas 
$U \to X$ of Krull dimension $n$.  
\end{defn}

A stack $X$ is called a {\it tame} stack if for any geometric point
$s : {\rm Spec}(\Omega) \to X$, the group ${\rm Aut}_{\Omega}
(s)$ has order prime to the characteristic of the field $\Omega$.
In particular, if $G$ is an algebraic group acting on a scheme $X$ of
finite type over a field $k$ such that the stabilizers of geometric
points are finite and reduced, then the quotient stack $[X/G]$
is tame if and only if the orders of stabilizer groups are prime to
the characteristic of $k$. Thus, all stacks over a field of characteristic
zero are tame.
 
The following lemma about the separated Deligne-Mumford stacks is
well-known.

\begin{lem}[{\cite[Lemma~2.2.3]{Ab-V}}]\label{lem:moduli}
Let $X$ be a noetherian and separated Deligne-Mumford stack. Then there
is an \'etale cover $\{X_{\alpha} \to X\}_{\alpha \in I}$ such that each $X_{\alpha}$
is of the form $X_{\alpha} = [U_{\alpha}/{\Gamma_{\alpha}}]$, where $U_{\alpha}$ is
an affine noetherian scheme and $\Gamma_{\alpha}$ is a finite group acting on
$U_{\alpha}$.
\end{lem}
\begin{proof}(Sketch)
It follows from \cite[Theorem~6.12]{DR} that $X$ has a coarse moduli space
$q : X \to Y$ which is a noetherian and separated algebraic space.
Choose a  geometric point $y_{\alpha}$ of $Y$ and let $Y^{sh}_{\alpha}$ denote the 
spectrum of the strict henselization of $Y$ at the point $y_{\alpha}$. 
In particular, $Y^{sh}_{\alpha}$ is a noetherian affine scheme. It is then shown 
in {\sl loc. cit.} that there is a scheme $U_{\alpha}$ which is finite over 
$Y^{sh}$ and a finite group $\Gamma_{\alpha}$ acting on $U_{\alpha}$ such that 
$q^{-1}(Y^{sh}_{\alpha}) = [{U_{\alpha}}/{\Gamma_{\alpha}}]$ and $Y^{sh}_{\alpha} = 
{U_{\alpha}}/{\Gamma_{\alpha}}$. Since $U_{\alpha} \to Y^{sh}_{\alpha}$ is finite,
we see that $U_{\alpha}$ is a noetherian affine scheme. Since $Y^{sh}_{\alpha}$
is an inverse limit of affine \'etale covers, we see that there is 
an \'etale cover $\{Y_{\alpha} \to Y\}$ by noetherian affine schemes such
that for each $\alpha \in I$, one has $q^{-1}(Y_{\alpha}) = 
[U_{\alpha}/{\Gamma_{\alpha}}]$ with $U_{\alpha}$ noetherian and affine.
\end{proof}  

\subsection{Sheaves on Deligne-Mumford stacks}\label{sec:sheaves}
We recall from \cite[12.1]{LMB} that the (small) \'etale site
\'Et$(X)$ of a Deligne-Mumford stack is the category whose objects
are the representable \'etale morphisms $u : U \to X$ with $U$
a scheme (denoted by $(U, u)$) and morphisms between $(U, u)$
and $(V,v)$ are 1-morphisms of schemes $f : U \to V$ such that there is
a 2-isomorphism $\alpha : v \circ f \to u$. The Grothendieck topology on 
\'Et$(X)$ is generated by the
coverings Cov$(U,u)$ whose objects are families of morphisms
${\alpha}_i : (U_i, u_i) \to (U, u)$ such that the morphism 
\[
{\coprod}_i u_i : {\coprod}_i U_i \to U
\]
of schemes is surjective.

A sheaf (resp. presheaf) on $X$ is a sheaf (resp. presheaf) on the
\'etale site \'Et$(X)$ in the sense of \cite{Artin}.
In particular, a sheaf of ${\sO}_X$-modules (resp. quasi-coherent
sheaf) $\sF$ on $X$ is the following data.\\
$(i)$ \ For each atlas $u : U \to X$, an ${\sO}_U$-module (resp.
quasi-coherent sheaf) ${\sF}_U$ on $U$. \\
$(ii)$ \ For any morphism of atlases $f : (U, u) \to (V, v)$,
an isomorphism ${\sF}_U \xrightarrow {\cong} f^*({\sF}_V)$. 
The above isomorphisms are required to satisfy the usual cocycle
conditions. 

\begin{remk}\label{remk:topos}
We shall often consider the bigger \'etale site ${\sE}t(X)$, whose objects
are all morphisms $[X' \xrightarrow{f} X]$, where $X'$ is a Deligne-Mumford
stack and $f$ is an \'etale morphism (not necessarily representable).
The morphisms in this site are defined as in \'Et$(X)$.  Note that
every covering in the Grothendieck site ${\sE}t(X)$ has refinements
which are from \'Et$(X)$. It is then 
easy to see that the sites \'Et$(X)$ and ${\sE}t(X)$ have the same topos.
We shall denote this common \'etale topos by {\sl \'et}$(X)$.
\end{remk}
 
A coherent sheaf is a quasi-coherent sheaf $\sF$ such
that ${\sF}_U$ is coherent for each atlas $U \to X$. A vector
bundle is a coherent sheaf $\sF$ such that ${\sF}_U$ is 
locally free for each atlas $U \to X$. Let $Sh(X)$ (resp.
$Mod(X)$) (resp. $QC(X)$) denote the abelian category of sheaves
of abelian groups (resp. sheaves of ${\sO}_X$-modules)
(resp. quasi-coherent sheaves) on $X$.

For any sheaf of abelian groups $\sF$, we define the global
section functor ${\Gamma}(X, {\sF})$ as Hom$({\Z}_X, {\sF})$.
This is a left exact functor and its right derived functors are
denoted by $H^i(X, {\sF})$ for $i \ge 0$. It is easy to check that
if $\sF$ is a sheaf of ${\sO}_X$ modules, then ${\Gamma}(X, {\sF})
= {\rm Hom}_{{\sO}_X}({\sO}_X, {\sF}) = \Gamma \left(S, \pi_*(\sF)\right)$,
where $\pi : X \to S$ is the structure morphism. Moreover, $H^i(X, -)$ is same
as the derived functors of ${\rm Hom}_{{\sO}_X} ({\sO}_X, -)$ computed 
in the category $Mod(X)$. The following result was proven for the
special case of strongly representable maps in \cite[Lemma~1.6]{Krishna}.  

\begin{lem}\label{lem:shrik}
Let $X$ be a stack and let $j: Y \to X$ be an \'etale 
morphism of stacks. Then there is a functor $j_{!}: Sh(Y) \to Sh(X)$ 
which is exact and is a left adjoint to the restriction functor $j^*$. 
Furthermore, $j_{!}$ takes ${\sO}_Y$-modules to ${\sO}_X$-modules. 
In particular, $j^*$ preserves limits and injective sheaves of 
${\sO}_X$-modules.
\end{lem}
\begin{proof} Except for the last assertion, this result is already shown
in \cite[Lemma~1.6]{Krishna}. The last assertion is also shown there for
the strongly representable maps. We just have to show that the general case 
can be reduced from the case of strongly representable maps. 
So let $(U,u) \in$ \'Et$(X)$ and let ${\sF} \in Mod(Y)$.
It is enough to show that $u^* \circ j_{!} (\sF)$ is an 
${\sO}_U$-module. Consider the Cartesian diagram
\[
\xymatrix{
V \ar[r]^{j'} \ar[d]_{u'} & U \ar[d]^{u} \\
Y \ar[r]_{j} & X.
}
\]
We now note from {\sl loc. cit.} that 
\begin{equation}\label{eqn:shrik1}
j_{!} (\sF)((U,u)) = {\underset{\{u': U \to Y | j \circ u' = u\}}\bigoplus}
{\sF}((U, u')),
\end{equation}
from which it is easy to see that $u^* \circ j_{!} (\sF)
= {j'}_{!} \circ {u'}^*(\sF)$. This reduces the problem to the
case when $X$ is a scheme. We now let $V \xrightarrow{v} Y$ 
an \'etale atlas. There is a natural map $v_{!}(v^*(\sF)) \to \sF$
and $v_{!}(v^*(\sF))$ is a sheaf of $\sO_Y$-modules since $v$ is
strongly representable.  Since $v^* \circ v_{!}$ is identity, as
follows from ~\eqref{eqn:shrik1}, we see that this map is an isomorphism
when restrricted to $V$. Since $V \to Y$ is an atlas, this map must be
an isomorphism.
Thus we can assume that $\sF = v_{!}(\sG)$, where $\sG$ is a sheaf of
$\sO_V$-modules. We have then $j_{!}(\sF) = j_{!} \circ v_{!}(\sG) =
(j \circ v)_{!}(\sG)$ and this reduces to the case when $X$ is a scheme and
$Y$ is an \'etale open subset of $X$. In this case, the definition in
~\eqref{eqn:shrik1} agrees with the definition of $j_{!}$ in 
\cite[Remark~3.18]{Milne} for schemes, which is known to preserve the sheaves 
of $\sO_X$-modules.
\end{proof}
\begin{lem}{$($\cite[Lemma~1.7]{Krishna}$)$}\label{lem:injective}
Let $X$ be a stack. Then $QC(X)$ and $Mod(X)$ are 
Grothendieck categories and hence have enough injectives
and all limits.
\end{lem}

\begin{proof} Since there is a minor error in the proof of
\cite[Lemma~1.7]{Krishna}, we reproduce the complete proof here.
Recall that an abelian category is a Grothendieck
category if it has inductive limits, a family of generators
and if the filtered inductive limits are exact. The first
and the last conditions are easily verified for the given
categories. So we only need to show that both the categories
have families of generators. Since every quasi-coherent sheaf on $X$ is
a direct limit of a family of its coherent subsheaves ({\sl cf.} 
\cite[Proposition~15.4]{LMB}), the family of the isomorphism classes
of coherent sheaves clearly forms a generating family for the
category $QC(X)$ ({\sl cf.} \cite[B.3]{TT}). 

To give a generating family for $Mod(X)$, let $u : U \to X$ be
an atlas. We claim that $\{(u \circ j)_{!} ({\sO}_V) |
V \xrightarrow {j} U \ {\rm open} \}$ is the required family.
As before, let ${\sF} \inj {\sG}$ be an inclusion in $Mod(X)$
such that ${\sG}/{\sF}$ is not zero. Then there is an open subset
$V \xrightarrow {j} U$ such that this condition remains true
when we restrict these sheaves to $V$. Now we can further
restrict $V$ to ensure that there is map ${\sO}_V \to
(u \circ j)^*(\sG)$ such that the composite ${\sO}_V \to
(u \circ j)^*({\sG}/{\sF})$ is not zero. Now we use 
Lemma~\ref{lem:shrik} to get a map 
$(u \circ j)_{!} ({\sO}_V) \to {\sG}$ such that the composite 
$(u \circ j)_{!} ({\sO}_V) \to {\sG}/{\sF}$ is not zero.
This proves the claim. 

Finally, it is known that a Grothendieck category has enough
injectives and the Gabriel-Popescu embedding theorem implies
that such categories have all limits.
\end{proof} 

\begin{remk}\label{remk:sheaves}
It is well known (and easy to check) that the category $Sh(X)$ is also a 
Grothendieck category ({\sl cf.} \cite{Artin}, Theorem~1.6).
\end{remk}

\section{Perfect complexes on stacks}\label{section:Perfect}
Let $X$ be a stack. Let $C(qc/X)$ and $C_{qc}(X)$ respectively denote
the categories of (possibly unbounded) cochain complexes of quasi-coherent 
sheaves and cochain complexes of ${\sO}_X$-modules with quasi-coherent 
cohomology. Let $D(qc/X)$ and $D_{qc}(X)$ denote the corresponding derived
categories. Note that $D_{qc}(X)$ is a full subcategory of
the derived category $D(X)$ of all ${\sO}_X$-modules. Let 
$C(X)$ denote the category of unbounded cochain complexes
of ${\sO}_X$-modules. If $K$ is a complex of ${\sO}_X$-modules,
we denote its cohomology sheaves by ${\sH}^i(K)$. Recall that a stack $X$ is
said to have the {\sl resolution property} if every coherent sheaf on $X$
is a quotient of a locally free sheaf. It follows from 
the results of Thomason \cite{Thomason1} that the quotient stack
$[X/G]$ for a linear action of an algebraic group $G$ on a quasi-projective
variety $X$ over a field has the resolution property. In particular,
this property is satisfied for practically all geometric stacks.

For a Grothendieck category $\sA$, let $C(\sA)$ denote the abelian
category of all (possibly unbounded) complexes of objects in $\sA$ and
let $D(\sA)$ denote its derived category.
Spaltenstein \cite{Spalten} defined a complex $I$ over
$\sA$ to be $K$-injective if for every acyclic complex $J$, the
complex of abelian groups ${\rm Hom}^{\bullet} (J, I)$ is acyclic.
This is equivalent to saying that in the homotopy category, there
are no nonzero morphisms from an acyclic complex to $I$. Serpe
\cite{Serpe} has shown that every unbounded complex over $\sA$ 
has a $K$-injective resolution. On the other hand, since $\sA$
has enough injectives, a complex over $\sA$ also has
a Cartan-Eilenberg resolution ({\it cf.} \cite[Appendix~A]{Keller}).
It is known \cite{Weibel} that a Cartan-Eilenberg 
resolution of an unbounded complex over $\sA$ need not give a
$K$-injective resolution. In other words, the Cartan-Eilenberg
hypercohomology may not coincide with the derived functor cohomology
for complexes in a Grothendieck category. The following result was
proven in \cite[Proposition~2.2, Corollary~2.6]{Krishna} whose scheme version 
was shown by Keller in \cite[Apendix~A]{Keller}.

\begin{thm}\label{thm:CER}
Let $X$ be a stack of dimension $n$ and let $K \in C_{qc}(X)$. Let
$K \xrightarrow {\epsilon} I^{\bullet, \bullet}$ be a
Cartan-Eilenberg resolution of $K$ in $C(X)$. Then $K \xrightarrow
{\epsilon} {\widehat {\rm Tot}} I$ is a $K$-injective resolution
of $K$, where ${\widehat {\rm Tot}} I$ is the product total complex of
$I^{\bullet, \bullet}$. If $X$ is a tame stack with the resolution property,
then every Cartan-Eilenberg resolution of a complex $K \in C(qc/X)$
is a $K$-injective resolution. 
\end{thm}

Let $X$ be a stack and let $E$ be a cochain 
complex of ${\sO}_X$-modules with quasi-coherent cohomology. 
\begin{defn}[\cite{Joshua}]\label{defn:pseudocoherent}
We say that $E$ is {\sl pseudo-coherent} if all cohomology sheaves
$H^i(E)$ are coherent, and $E$ is cohomologically bounded above, i.e.,
$H^i(E) = 0$ for $i \gg 0$.
\end{defn}

\begin{defn}\label{defn:perfect}
We say that $E$ is {\sl strictly perfect} if it is a bounded complex of
vector bundles on $X$. We say that $E$ is {\sl perfect} if there is an
atlas $U \xrightarrow {u} X$ such that $u^*(E)$ is isomorphic to
a strictly perfect complex in $D(U)$.
\end{defn}

The following properties of perfect complexes 
were proven in \cite[Proposition~4.3]{Krishna}. We also refer the 
reader to \cite[Proposition~2.9]{Joshua2}
for some other properties of perfect complexes on more general stacks.
\begin{prop}\label{prop:properties}
Let $X$ be as above and let $E \in C_{qc}(X)$. \\
$(i)$ \ If $E$ is perfect, then it is pseudo-coherent and cohomologically 
bounded. \\
$(ii)$ \ $E$ is perfect if and only if there is an atlas $V \xrightarrow
{v} X$ and a strictly perfect complex $F$ on $V$ with a morphism
$F \to v^*(E)$ which is a quasi-isomorphism. \\
$(iii)$ \ If $X$ is a scheme, then $E$ is perfect if and only if
it is perfect in the sense of \cite{TT}. \\
$(iv)$ \ If $E \cong F$ in $D_{qc}(X)$, then $E$ is perfect if and 
only if $F$ is so. \\
$(v)$ \ If $E = E_1 \oplus E_2$, then $E$ is perfect if and only if
$E_1$ and $E_2$ are so.
\end{prop}
Our aim in this and the next section is to characterize the perfect 
complexes on stacks in terms of the full subcategory of {\sl compact} objects 
in certain triangulated category. We begin with a brief recall of the relevant 
notions. Recall 
that for a triangulated category $\sC$ admitting arbitrary direct sums, an 
object $E$ in $\sC$ is {\sl compact} if ${\rm Hom}_{\sC}(E, -)$ commutes with 
direct sums. Let $\sC^c$ be the full subcategory of $\sC$ consisting of compact
objects. Then $\sC^c$ is a triangulated category and $\sC$ is called
{\sl compactly generated} if $\sC$ is generated by $\sC^c$. If 
$\sE = (A_{\lambda})_{\lambda \in I}$ is a set of objects in $\sC$, then we say that
$\sE$ {\sl classically generates} $\sC$ if the smallest thick triangulated
subcategory of $\sC$ containing $\sE$ is equal to $\sC$ itself. We say that
$\sC$ is finitely generated if it is classically generated by one object. 
We shall need the following results in the next section.

%\begin{thm}\label{thm:Neeman-R}
%Assume that $\sC$ is a compactly generated triangulated category. Then a set
%of objects $\sE \subset {\sC}^c$ classically generates ${\sC}^c$ if and only
%if it generates $\sC$.
%\end{thm}
\begin{lem}\label{lem:K-injAD}
Let $F : \sA \to \sB$ be a functor between Grothendieck categories which
has an exact left adjoint. Then $F : C(\sA) \to C(\sB)$ preserves the 
$K$-injective complexes.
\end{lem}
\begin{proof} Let $G$ be the exact left adjoint of $F$.
Let $A$ be a $K$-injective complex in $C(\sA)$ and let $B$
be an acyclic complex in $C(\sB)$. Then the exactness of $G$ implies that
$G(B)$ is acyclic. Now the adjointness of the pair $(G, F)$ implies that
the complex ${\rm Hom}^{\bullet}_{\sB}\left(B, F(A)\right)$ is canonically
isomorphic to the complex ${\rm Hom}^{\bullet}_{\sA}\left(G(B), A\right)$,
and this latter complex is acyclic because $G(B)$ is acyclic and $A$ is
$K$-injective.
\end{proof}  

\begin{cor}\label{cor:K-injAD*}
Consider the Cartesian diagram of stacks
\[
\xymatrix@C.6pc{
W \ar[r]^{t} \ar[d]_{g}  & Y \ar[d]^{f} \\
Z \ar[r]_{s} & X,}
\]
where $s$ is \'etale. For any $A \in D(Y)$, the natural
map $s^*Rf_*(A) \to Rg_*t^*(A)$ is an isomorphism in $D(Z)$.
\end{cor}
\begin{proof} Since $s$ is \'etale, we see that $t$ is also \'etale and hence
$s^*$ and $t^*$ are exact functors on $Mod(X)$ and $Mod(Y)$ respectively.
It follows from Lemmas~\ref{lem:shrik}, ~\ref{lem:injective} and 
~\ref{lem:K-injAD} that $s^*$ and $t^*$ preserve $K$-injective complexes. 
To prove the corollary, we can assume that $A$ is $K$-injective.
Then we have $s^*Rf_*(A) = s^*f_*(A)$ and $Rg_* t^*(A) = g_*t^*(A)$.
Thus we only need to show that for a sheaf $\sF$ of $\sO_Y$-modules, the
natural map $s^*f_*(\sF) \to g_*t^*(\sF)$ is an isomorphism.
But this is immediate from the definition of the push-forward and the
pull-back maps on the \'etale topoi of the underlying stacks.
\end{proof}
\section{characterization of perfect complexes on stacks}
\label{section:PFC}
In this section, we prove our main result on the characterization of the 
category of perfect complexes on a separated Deligne-Mumford stack $X$ in 
terms of the full subcategory of the compact objects in $D_{qc}(X)$. We begin 
with the following result. 
\begin{prop}\label{prop:direct-image}
Let $f : Y \to X$ be a strongly representable morphism in ${\sD}{\sM}_S$. Then 
$Rf_*$ maps $D_{qc}(Y)$ into $D_{qc}(X)$ and it commutes with arbitrary direct 
sums on $D_{qc}(Y)$.
\end{prop}
\begin{proof} 
Let $U \xrightarrow{u} X$ be a noetherian atlas for $X$ and consider the
Cartesian diagram
\begin{equation}\label{eqn:atlas}
\xymatrix@C.8pc{
V \ar[r]^{v} \ar[d]_{g}  & Y \ar[d]^{f} \\
U \ar[r]_{u} & X.}
\end{equation}
Since $f$ is strongly representable, we see that $g$ is a morphism of
noetherian schemes. To show that $Rf_*(A) \in D_{qc}(X)$ for $A \in D_{qc}(Y)$,
we can assume that $A$ is $K$-injective, and then it suffices to show that 
$u^*f_*(A)$ has quasi-coherent cohomology sheaves. However, it follows from
Corollary~\ref{cor:K-injAD*} that $u^*f_*(A) \xrightarrow{\cong}
g_*v^*(A)$ and so we can assume that $f$ is a morphism of noetherian
schemes, in which case this was shown in \cite[Theorem~3.3.3]{Bo-Va}.
This proves the first part of the proposition.

To show the second part, let $\left(A_{\lambda}\right)$ be a family of
objects in $D_{qc}(Y)$ and let $A =  {\underset {\lambda} \oplus} A_{\lambda}$.
It suffices again to show that
the natural map $u^*\left({\underset {\lambda} \oplus} Rf_* A_{\lambda}\right)
\to u^*Rf_*(A)$ is an isomorphism on any \'etale atlas $U$ of $X$.
However, we have
\[
\begin{array}{lll}
u^*Rf_*(A) & \cong & Rg_*v^*(A) \ \ \ \ \ \ \ \
({\rm by \ Corollary~\ref{cor:K-injAD*}}) \\
& \cong & Rg_*\left({\underset {\lambda} \oplus} v^* A_{\lambda}\right)
\\
& \cong &  {\underset {\lambda} \oplus} Rg_* v^* A_{\lambda} \\
& \cong & {\underset {\lambda} \oplus} u^*Rf_* A_{\lambda} \ \ \ \ \ \ \
({\rm by \ Corollary~\ref{cor:K-injAD*}})\\
& \cong & u^*\left({\underset {\lambda} \oplus} Rf_* A_{\lambda}\right),
\end{array}
\]
where the second and the last isomorphisms follow from the fact that
$V \in$ \'Et$(Y)$ and $U \in$ \'Et$(X)$ and the
third isomorphism follows from \cite[Corollary~3.3.4]{Bo-Va}
as $g$ is a morphism of noetherian schemes and hence quasi-compact and
separated by Proposition~\ref{prop:QCS}.
This completes the proof.
\end{proof}
\begin{cor}\label{cor:restriction}
Let $X$ be a stack and let $U \overset{j}{\inj} X$ be an open substack with
complement $Z \overset{i}{\inj} X$. Let $p: V \to X$ be a
representable and \'etale morphism of stacks and let $T = p^{-1}(Z)$. 
Then $p^*: D_{qc}(X \ {\rm on} \ Z) \to D_{qc}(V \ {\rm on} \ T)$ preserves
perfect and compact objects.
\end{cor}
\begin{proof} Since $p$ is \'etale and representable, it is in fact
strongly representable by Proposition~\ref{prop:QCS} and 
\cite[Corollary~II.6.17]{Knutson}.
Since it is immediate from the definitions that an \'etale map
preserves perfect complexes, we only need to show that $p$ preserves the
compact objects. Suppose that $A \in D_{qc}(X \ {\rm on} \ Z)$ is compact.
Let $W = p^{-1}(U)$ and consider the Cartesian diagram
\begin{equation}\label{eqn:restriction1}
\xymatrix@C.8pc{
W \ar[r]^{j'} \ar[d]_{p'} & V \ar[d]^{p} \\
U \ar[r]_{j} & X.}
\end{equation}
Let $(A_{\lambda})$ be a family of objects in $D_{qc}(V \ {\rm on} \ T)$.
We have then
\begin{equation}\label{eqn:restriction2} 
\begin{array}{lll}
{\rm Hom}_{D_{qc}(V \ {\rm on} \ T)} 
\left(p^*A, {\underset {\lambda} \oplus} A_{\lambda}\right)
& = & 
{\rm Hom}_{D_{qc}(V)} 
\left(p^*A, {\underset {\lambda} \oplus} A_{\lambda}\right) \\
& = & 
{\rm Hom}_{D_{qc}(X)}
\left(A, Rp_*\left({\underset {\lambda} \oplus} A_{\lambda}\right)\right) \\
& = & {\rm Hom}_{D_{qc}(X)}
\left(A, {\underset {\lambda} \oplus} Rp_*A_{\lambda}\right) 
\hspace*{.8cm} ({\rm by \ Proposition~\ref{prop:direct-image}}) \\
& = & {\rm Hom}_{D_{qc}(X \ {\rm on} \ Z)} 
\left(A, {\underset {\lambda} \oplus} Rp_*A_{\lambda}\right) \\
& = & {\underset {\lambda} \oplus} {\rm Hom}_{D_{qc}(X \ {\rm on} \ Z)} 
\left(A, Rp_*A_{\lambda}\right) \\
& = & {\underset {\lambda} \oplus} {\rm Hom}_{D_{qc}(X)} 
\left(A, Rp_*A_{\lambda}\right) \\
& = & {\underset {\lambda} \oplus} {\rm Hom}_{D_{qc}(V)} 
\left(p^*A, A_{\lambda}\right) \\
& = & {\underset {\lambda} \oplus} {\rm Hom}_{D_{qc}(V \ {\rm on} \ T)} 
\left(p^*A, A_{\lambda}\right). 
\end{array}
\end{equation}
Here, the second and the one before the last equality follow from 
Lemma~\ref{lem:injective} and \cite[Lemma~3.3]{Krishna} since
$(p^*, p_*)$ is a pair of adjoint functors with $p^*$ exact.
The fourth equality follows from Corollary~\ref{cor:K-injAD*} since 
$j$ is representable and \'etale and the fifth equality holds because
$A$ is compact in $D_{qc}(X \ {\rm on} \ Z)$. This shows the desired
compactness of $p^*A$.
\end{proof}
The following result is a generalization of \cite[Theorem~3.1.1]{Bo-Va}.
\begin{cor}\label{cor:compSc}
Let $X$ be a noetherian and separated scheme and let $U \overset{j} {\inj}
X$ be an open subscheme with complement $Z \overset{i} {\inj} X$. Let 
$D_{qc}(X \ {\rm on} \ Z)$ denote the full subcategory of $D_{qc}(X)$ 
consisting of complexes which are acyclic over $U$. Then an object of
$D_{qc}(X \ {\rm on} \ Z)$ is compact if and only if it a perfect complex
on $X$ which is acyclic over $U$.
\end{cor}
\begin{proof} If $A \in D_{qc}(X \ {\rm on} \ Z)$ is perfect, then it is
also perfect as an object of $D_{qc}(X)$ and hence compact by 
\cite[Theorem~3.1.1]{Bo-Va}. But an object of $D_{qc}(X \ {\rm on} \ Z)$ which
is compact in $D_{qc}(X)$ is also compact as an object of
$D_{qc}(X \ {\rm on} \ Z)$.

Suppose now that $A \in D_{qc}(X \ {\rm on} \ Z)$ is compact. It follows from
Corollary~\ref{cor:restriction} that the restriction of $A$ to any given
affine open subscheme $V \inj X$ is compact in 
$D_{qc}(V \ {\rm on} \ V \cap Z)$.
Since $V$ is noetherian and affine, it has the resolution property (in fact 
has an ample line bundle) and hence $A$ is perfect on $V$ by 
\cite[Proposition~5.4.2]{TT} (see also  \cite{NR}). Hence $X$ has a finite
affine open cover $\{X_{\alpha} \inj X\}$ such that the restriction of
$A$ to each $D_{qc}(X_{\alpha} \ {\rm on} \ X_{\alpha} \cap Z)$ is perfect. We 
conclude that $A$ must be perfect on $X$ which is acyclic over $U$.
\end{proof} 
The following are our main results of this section.
\begin{thm}\label{thm:Perf-comp}
Let $X$ be stack and let $U \overset{j}{\inj} X$ be an open substack with
complement $Z \overset{i}{\inj} X$. Then every compact object in
$D_{qc}(X \ {\rm on} \ Z)$ is perfect.
If $X$ is a tame stack with a coarse moduli scheme, then every perfect 
complex in $D_{qc}(X \ {\rm on} \ Z)$ is compact.
\end{thm}
\begin{proof}
We first assume that $A$ is compact in $D_{qc}(X \ {\rm on} \ Z)$ and show that
it is perfect.
Let $U \xrightarrow{u} X$ be a noetherian atlas of $X$. We have seen before 
that $u$ is strongly representable. We can then apply 
Corollary~\ref{cor:restriction} to conclude that $u^*A$ is a compact
object of $D_{qc}(U \ {\rm on} \ u^{-1}(Z))$. 
It subsequently follows from Corollary~\ref{cor:compSc} 
that $u^*A$ is a perfect complex on the scheme $U$. We now apply 
Proposition~\ref{prop:properties} ($iii$) to get an \'etale cover
$V \xrightarrow{v} U$ such that $v^*u^*A$ is quasi-isomorphic to a 
strictly perfect complex on $V$. Since the composite $V \to X$ is an
\'etale atlas of $X$, we conclude that $A$ is a perfect complex $X$ which
is acyclic over $U$.

To prove that a perfect complex $A$ in $D_{qc}(X \ {\rm on} \ Z)$ is a
compact object of $D_{qc}(X \ {\rm on} \ Z)$, we note that $A$ is then a
perfect complex on $X$. Hence it is enough to prove the case when $Z =
\emptyset$. 
So we assume now that $X$ is a tame stack and has a coarse moduli scheme 
$q : X \to Y$. Then $Y$ is a noetherian and separated scheme by 
\cite[Theorem~6.12]{DR}.
It follows from Lemma~\ref{lem:moduli} that $Y$ has a finite affine cover
$\{Y_i \to Y\}_{1 \le i \le n}$ such that for each $i$, the diagram 
\begin{equation}\label{eqn:moduli*}
\xymatrix@C.8pc{
[U_i/{\Gamma_i}] \ar[r]^{j'_i} \ar[d]_{p_i} & q^{-1}(Y_i) \ar[d]^{q_i} \\
V_i \ar[r]_{j_i} & Y_i}
\end{equation}
is Cartesian, $U_i$ and $V_i$ are affine noetherian schemes and the
horizontal maps are \'etale and surjective and $[U_i/{\Gamma_i}]$ is a tame 
stack. It follows that each $p_i$ is a coarse moduli space.

To show the compactness of an object, we first observe from the above
definition of the compactness that an object $A \in D_{qc}(X)$ is compact
if and only if given any family $(A_{\lambda})$, any morphism
$A \to {\underset {\lambda}\oplus} A_{\lambda}$ factors through the morphism
into a finite direct sum of $A_{\lambda}$'s. From this criterion, it is
easy to see that $A$ is compact if there is a finite Zariski open cover
$X_{i} \inj X$, such that the restriction of $A$ to each open substack
$X_i$ is compact. 

To prove now the second part of the theorem when $Z = \emptyset$, let $A$ be a 
perfect complex on $X$. From the above observation, it suffices to show that 
the restriction of $A$ to each $X_i = q^{-1}(Y_i)$ is compact, where 
$\{Y_i \inj Y\}_{1 \le i \le n}$ is an affine open cover as in ~\eqref{eqn:moduli*}.
Since the perfectness is clearly preserved under the restriction to an open 
substack, we see that $A|_{X_i}$ is perfect for each $i$.
Thus, we are reduced to proving the theorem when $X$ a tame stack with a
coarse moduli scheme $q : X \to Y$ such that $Y$ is noetherian and affine
and there is a Cartesian diagram
\begin{equation}\label{eqn:moduli*1}
\xymatrix@C.8pc{
[U/{\Gamma}] \ar[r]^{j'} \ar[d]_{p} &  X \ar[d]^{q} \\
V \ar[r]_{j} & Y,}
\end{equation}
where the horizontal morphisms are \'etale covers, $U$ and $V$ are noetherian
and affine schemes and $\Gamma$ is a finite group acting on $U$ such that
$X' = [U/{\Gamma}]$ is a tame stack. The map
$X' \xrightarrow{p} V$ is then a coarse moduli scheme.

We now note that as a locally presentable dg-category in the sense of
\cite[Section~3]{Toen}, the category $C_{qc}(X)$ is simply the unit 
dg-category $\bf{1}$ over $X$. Let $\alpha$ be the locally presentable
dg-category on $Y$ defined by 
\[
(W \to Y) \mapsto C_{qc}\left(q^{-1}(W)\right).
\]
It is then clear that in the notations of \cite{Toen}, 
$C_{qc}(X) = L_{\bf{1}}(X) = L_{\alpha}(Y)$. Thus, it suffices to show that
$A$ is a compact object in $L_{\alpha}(Y)$. Since $V \xrightarrow{j} Y$
is an \'etale and surjective map between affine schemes, it suffices to show 
using \cite[Lemma~3.4]{Toen} that $j^*(A)$ is compact in $L_{\alpha}(V)$.
However, we see as above that $L_{\alpha}(V) = L_{j'^{-1}(\bf{1})}(X')
= L_{\bf{1}}(X') = C_{qc}(X')$. In particular, it suffices to show that
$j'^*(A)$ is compact in $D_{qc}(X')$.

Since $j'$ is \'etale, we see that $A' = j'^*(A)$ is a perfect complex
on $X'$. Since $U$ is an affine noetherian scheme, it clearly has the
resolution property. Since $\Gamma$ is a finite group acting on $U$,
it follows from \cite[Lemma~2.14]{Thomason1} that $X'$ also has the resolution
property. The tameness of $[U/{\Gamma}]$ and \cite[Lemma~3.5]{Krishna} imply
that there is an equivalence of categories $D(qc/{X'}) 
\xrightarrow{\cong} D_{qc}(X')$.
It follows now from \cite[Theorem~4.12]{Krishna} that $A'$ is compact
in $D_{qc}(X')$. This completes the proof of the theorem.
\end{proof} 
\begin{thm}[\bf{Localization}]\label{thm:Localization}
Let $X$ be a tame stack which admits a coarse moduli scheme. Let $U \inj X$
be an open substack with the complement $Z \inj X$. Then there is a
homotopy fibration sequence of non-connective spectra
\[
K(X \ {\rm on} \ Z) \to K(X) \to K(U).
\]
In particular, one has a long exact sequence 
\[
\cdots \to K_{i}(X \ {\rm on} \ Z) \to K_i(X) \to K_i(U) \to 
K_{i-1}(X \ {\rm on} \ Z) \to \cdots
\]
for all $i \in \Z$.
\end{thm}
\begin{proof} This follows immediately from Theorem~\ref{thm:Perf-comp}
and \cite[Corollary~4.3]{Toen}.
\end{proof}
\section{Excision for $K$-theory of stacks}\label{section:Loc-Exc}
In this section, we study the excision
property of the $K$-theory of perfect complexes on stacks.
For a stack $X$ and a closed substack $Z \inj X$, let $C_Z(Perf/X)$ denote the 
dg-category of perfect complexes on $X$ which are acyclic on the complement
of $Z$. Let $D_Z(Perf/X)$ denote the corresponding derived category.
Then $D_Z(Perf/X)$ is a full triangulated subcategory of 
$D_{qc}(X \ {\rm on} \ Z)$. We write $D(Perf/X)$ for the full triangulated
subcategory of perfect complexes on $X$. Let $K(X \ {\rm on} \ Z)$ denote the 
non-connective $K$-theory spectrum of the dg-category $C_Z(Perf/X)$. 
Let $K(X)$ denote the non-connective $K$-theory spectrum of the perfect
complexes on $X$. We refer the reader to \cite{Cisinski} for details about the 
non-connective $K$-theory of dg-categories.

For a stack $X$, let $D^{-}_{qc}(X)$ and $D^{+}_{qc}(X)$ denote the full 
triangulated subcatgories of $D_{qc}(X)$ consisting of the bounded above
and bounded below complexes respectively. We define 
$D^{-}_{qc}(X \ {\rm on} \ Z)$ and $D^{+}_{qc} (X \ {\rm on} \ Z)$ in an obvious 
way. For a complex $E \in D(X)$, let
\[
{\tau}^{\le n}(E) = \left( \to E^i \to \cdots \to E^{n-1} \to Z^nE \to 0
\right) \ {\rm and} 
\]
\[
{\tau}^{\ge n}(E) = \left(0 \to E^n/{B^nE} \to E^{n+1} \to \cdots\right)
\]
denote the good truncations, where $B^nE$ and $Z^nE$ denote the boundaries
and the kernels of the differential of $E$. There is a natural isomorphism
\begin{equation}\label{eqn:ex0}
~{\underset {n}{\varinjlim}}~ {\tau}^{\le n}(E) \xrightarrow {\cong} E.
\end{equation}
\begin{equation}\label{eqn:ex1}
E \xrightarrow {\cong}  ~{\underset {n}{\varprojlim}}~ {\tau}^{\ge n}(E).
\end{equation}
\begin{lem}\label{lem:prop:ex2}
Let $f : Y \to X$ be a strongly representable morphism of stacks.
Then $Rf_*:D_{qc}(Y) \to D_{qc}(X)$ maps $D^{-}_{qc}(Y)$ into $D^{-}_{qc}(X)$ 
and $D^{+}_{qc}(Y)$ into $D^{+}_{qc}(X)$.
\end{lem}
\begin{proof} We first assume that $X$ is a scheme. The strong 
representability then implies that $Y$ is also a scheme. The first
assertion then follows from \cite[Theorem~3.3.3]{Bo-Va}.
To prove the second assertion, we can use Corollary~\ref{cor:K-injAD*}
to assume that $X$ is affine. If $\{Y_i \inj Y\}_{1 \le i \le n}$ is a finite
affine cover of $Y$, let $U = Y_2 \cup \cdots Y_n$ and $V = U \cap Y_1$.
Then the separatedness of $Y$ implies that $U$ and $V$ can be 
covered by at most $n-1$ affine open sets. Let 
$U \overset{j}{\inj} Y, Y_1 \overset{j_1}{\inj} Y$ and
$V \overset{j'}{\inj} Y$ denote the inclusion maps.
Then the exact triangle
\[
Rf_*E \to Rf_*(j^*E) \oplus Rf_*(j^*_1A) \to Rf_*(j'^*A) \to
\]
and an induction on the number of open sets in an affine cover of
$Y$ show that it suffices to prove the case when $Y$ is also affine.
But in this case, there is an equivalence $D^{+}(qc/X) 
\xrightarrow{\cong} D^{+}_{qc}(X)$ by \cite[B.16]{TT} and similarly
for $Y$. The lemma now follows in the scheme case since $f_*$ is exact on the 
category of quasi-coherent sheaves for maps between affine schemes.
In general, we can choose an atlas $U \to X$ and
use the strong representability of $f$ and Corollary~\ref{cor:K-injAD*}
to reduce to the case of schemes.
\end{proof}
\begin{prop}\label{prop:EXC1}
Let $f : Y \to X$ be a representable and \'etale morphism of stacks.
Let $Z \overset{i}{\inj} X$ be a closed
substack such that the restriction $f : Z {\times}_X Y \to Z$ is an
isomorphism. Then for any $E \in D^+_{qc}(X \ {\rm on} \ Z)$ and 
$F \in D^+_{qc}(Y \ {\rm on}\ Z {\times}_X Y)$, the natural maps 
$E \to Rf_* \circ f^*(E)$ and $f^* \circ Rf_*(F) \to F$ are isomorphisms.  
\end{prop}   
\begin{proof} We first prove the case when $X$ and $Y$ are schemes.
Put $W = Z {\times}_X Y$ and let $E \in D^+_{qc}(X \ {\rm on} \ Z)$. 
For a Cartan-Eilenberg resolution 
$E \xrightarrow {\epsilon} {\sI}^{\bullet \bullet}$ of $E$, 
we let ${\tau}^{\le n}\left({\sI}^{\bullet \bullet}\right)$ be the double
complex obtained by taking the good truncation of 
${\sI}^{\bullet \bullet}$ in each row. Put ${\sI} = {Tot}\left(
{\sI}^{\bullet \bullet}\right)$ and 
${\sI}^{[\le n]} = {Tot}\left(  
{\tau}^{\le n}\left({\sI}^{\bullet \bullet}\right)\right)$ for $n \in \Z$.
Then we get for any $i \in \Z$,
\[
{\sI}^i = ~{\underset {{\alpha}+ {\beta} = i}{\bigoplus}}~
\left({\sI}^{\alpha, \beta}\right) = {{\sI}^{[\le n]}}^i\]
for all large $n$. In particular, we get
\begin{equation}\label{eqn:ex02}
~{\underset {n}{\varinjlim}}~ {\sI}^{[\le n]} \xrightarrow {\cong} {\sI}.
\end{equation}
Lemma~\ref{lem:shrik} implies that $f^*(E) \xrightarrow 
{f^*(\epsilon)} f^*\left({\sI}^{\bullet \bullet}\right)$ is a Cartan-Eilenberg 
resolution of $f^*(E)$.
Since $E$ is bounded below, the product total complex of $I^{\bullet \bullet}$
(resp. $f^*\left(I^{\bullet \bullet}\right)$) coincides with the
corresponding direct sum total complex.
In particular, we get that ${Tot}\left(f^*\left({\sI}^{\bullet \bullet}\right)
\right) \cong f^*({\sI})$. Similarly, $f^*\left({\tau}^{\le n} (E) \right)
\to f^*\left({\tau}^{\le n}\left({\sI}^{\bullet \bullet}\right)\right)$
is a Cartan-Eilenberg resolution and ${Tot}
\left(f^*\left({\tau}^{\le n}\left({\sI}^{\bullet \bullet}\right)\right)
\right) \cong f^*\left({Tot}\left({\tau}^{\le n}\left({\sI}^{\bullet \bullet}\right)
\right)\right)$.
We conclude from this and Theorem~\ref{thm:CER} that the following are the 
$K$-injective resolutions of the complexes on the left. 
\[
E \to {\sI},  \ \  {\tau}^{\le n}(E) \to {\sI}^{[\le n]},
\]
\[
f^*(E) \to f^*(\sI), \ {\rm and} \ f^*\left({\tau}^{\le n}(E)\right)
\to f^*\left({\sI}^{[\le n]}\right).
\]
This implies in particular that $Rf_* \circ f^*\left(E\right) \cong
f_* \circ f^*\left({\sI}\right)$ and $Rf_* \circ f^*\left(
{\tau}^{\le n}(E)\right) \cong f_* \circ f^*\left({\sI}^{[\le n]}\right)$.
Combining this with ~\ref{eqn:ex0} and  ~\ref{eqn:ex02}, we get
\begin{equation}\label{eqn:ex03}
\begin{array}{lll} 
Rf_* \circ f^*\left(E\right) & \cong & f_* \circ f^*\left({\sI}\right) \\
& \cong &  f_* \circ f^*\left(~{\underset {n}{\varinjlim}}~ {\sI}^{[\le n]}
\right) \\
& \cong & ~{\underset {n}{\varinjlim}}~ f_* \circ f^*\left(
{\sI}^{[\le n]}\right) \\
& \cong &  ~{\underset {n}{\varinjlim}}~ Rf_* \circ f^*\left(
{\tau}^{\le n}(E)\right),
\end{array}
\end{equation}
where the third isomorphism follows, for example, by the exact triangle
\begin{equation}\label{eqn:ex03*}
~{\underset {n}{\oplus}}~ {\sI}^{[\le n]}  \to 
~{\underset {n}{\oplus}}~ {\sI}^{[\le n]} \to 
~{\underset {n}{\varinjlim}}~ {\sI}^{[\le n]}
\end{equation}
and Proposition~\ref{prop:direct-image}.
Thus, we get a commutative diagram 
\[
\xymatrix{
~{\underset {n}{\varinjlim}}~ {\tau}^{\le n}(E) \ar[r] \ar[d]_{\cong} &
~{\underset {n}{\varinjlim}}~ Rf_* \circ f^*\left(
{\tau}^{\le n}(E)\right) \ar[d]^{\cong} \\
E \ar[r] & Rf_* \circ f^*\left(E\right),}
\]
where the left vertical map is isomorphism by ~\eqref{eqn:ex0} and the
right vertical map is isomorphism by ~\eqref{eqn:ex03}. Since each 
${\tau}^{\le n}(E)$ is a good truncation of $E$, it is in
$D^+_{qc}(X \ {\rm on} \ Z)$. Moreover, as ${\tau}^{\le n}(E)$ is a bounded 
complex, the top horizontal map above is an isomorphism by 
\cite[Theorem~2.6.3]{TT}. We now conclude that
\[E \xrightarrow {\cong}  Rf_* \circ f^*\left(E\right).\]

Next we prove the second assertion of the proposition for schemes.
So let $F \in D^+_{qc}(Y \ {\rm on} \ W)$. Choosing a Cartan-Eilenberg 
resolution $F \xrightarrow {\epsilon} {\sI}^{\bullet \bullet}$, we have as
in the above case, the commutative diagram of isomorphisms
\[
\xymatrix{
~{\underset {n}{\varinjlim}}~ {\tau}^{\le n}(F) \ar[r]^{\hspace*{.7cm} \cong} 
\ar[d] & F \ar[d] \\
~{\underset {n}{\varinjlim}}~ {\sI}^{[\le n]} \ar[r]^{\hspace*{.7cm} \cong} &
{\sI},}
\]
where the complexes on the bottom are the $K$-injective resolutions of the
corresponding complexes on the top.
In particular, we get  
\begin{equation}\label{eqn:ex03}
\begin{array}{lll} 
Rf_*(F) & \cong & f_*\left({\sI}\right) \\
& \cong & ~{\underset {n}{\varinjlim}}~ f_*\left({\sI}^{[\le n]}\right) \\
& \cong & ~{\underset {n}{\varinjlim}}~ Rf_*\left({\tau}^{\le n}(F)\right),
\end{array}
\end{equation}
where the second isomorphism follows as before by ~\eqref{eqn:ex03*}.
Applying $f^*$ to above, we get
\begin{equation}\label{eqn:ex04}
~{\underset {n}{\varinjlim}}~ f^* \circ Rf_*\left({\tau}^{\le n}(F)\right) 
\xrightarrow {\cong}  f^* \left(~{\underset {n}{\varinjlim}}~
Rf_*\left({\tau}^{\le n}(F)\right)\right) \xrightarrow {\cong}
f^* \circ Rf_*(F).
\end{equation}
Thus we get a commutative diagram 
\[
\xymatrix{
~{\underset {n}{\varinjlim}}~ f^* \circ Rf_*\left({\tau}^{\le n}(F)\right)
\ar[r] \ar[d]_{\cong} &  ~{\underset {n}{\varinjlim}}~ {\tau}^{\le n}(F)
\ar[d]^{\cong} \\
f^* \circ Rf_*(F) \ar[r] & F,}
\]
where the left vertical arrow is an isomorphism by ~\eqref{eqn:ex04} and 
the right vertical arrow is an isomorphism by ~\eqref{eqn:ex0}.
Since each ${\tau}^{\le n}(F)$ is a bounded complex in 
$D^+_{qc}(Y \ {\rm on} \ W)$, the top horizontal arrow is an isomorphism by 
\cite[Theorem~2.6.3]{TT}.
We conclude that 
\[f^* \circ Rf_*(F) \xrightarrow {\cong} F.\]
This proves the proposition for noetherian schemes.

To complete the proof of the proposition, let $X$ be a stack and let
$u : U \to X$ be an {\'e}tale atlas for $X$ and consider the Cartesian
diagram ~\ref{eqn:atlas}. 
Put $T = u^{-1}(Z)$ and $P = v^{-1}(W)$. Since we have already seen that $f$ is 
strongly representable and \'etale, we see that $g$ is an \'etale morphism of 
noetherian and separated schemes which is an isomorphism over $T$.
It suffices to show that the natural maps
\[
u^*(E) \to u^*\left(Rf_*\circ f^*(E)\right) \ {\rm and} \
{v}^*\left(f^*\circ Rf_*(F)\right) \to {v}^*(F)
\]
are isomorphisms for $E \in  D^+_{qc}(X \ {\rm on} \ Z)$ and 
$F \in  D^+_{qc}(Y \ {\rm on} \ W)$. However, Corollary~\ref{cor:K-injAD*}
implies that 
$u^*\left(Rf_*\circ f^*(E)\right) \cong Rg_*\circ {v}^*\circ f^*(E)
\cong Rg_*\circ g^*\circ {u}^*(E)$, which in turn is isomorphic to
${u}^*(E)$ from the case of schemes proved above 
since $U$ is a noetherian scheme and ${u}^*(E) \in 
D^+_{qc}(U \ {\rm on} \ T)$.
Similarly, we have 
\[{v}^*\left(f^* \circ Rf_*(F)\right) \cong g^*\circ {u}^* \circ Rf_*(F)
\cong g^*\circ Rg_*\circ {v}^*(F),\]
again by Corollary~\ref{cor:K-injAD*}. The last term is isomorphic to
${v}^*(F)$, again from the case of schemes 
since $V$ is a noetherian scheme and $v^*(F) \in D^+_{qc}(V \ {\rm on} \ P)$.
\end{proof}
\begin{lem}\label{lem:EXC2}
Let $f : Y \to X$ be a representable and \'etale morphism of stacks.
Let $Z \overset{i}{\inj} X$ be a closed
substack such that the restriction $f : Z {\times}_X Y \to Z$ is an
isomorphism. Then for any $E \in D_{qc}(X \ {\rm on} \ Z)$, the natural map 
$E \to Rf_* \circ f^*(E)$ is an isomorphism.  
\end{lem}    
\begin{proof}
The isomorphism in ~\eqref{eqn:ex1} and Lemma~\ref{lem:shrik} imply that 
there are isomorphisms
\begin{equation}\label{eqn:ex11*}
E \xrightarrow {\cong} ~{\underset {n}{\varprojlim}}~ {\tau}^{\ge n} E, \ \
f^*(E) \xrightarrow {\cong}
~{\underset {n}{\varprojlim}}~ f^*\left({\tau}^{\ge n} E\right).
\end{equation}
Let $f^*(E) \to {\sI}^{\bullet \bullet}$ be a Cartan-Eilenberg resolution.
It follows from Theorem~\ref{thm:CER} that
$f^*(E) \to {\sI} = {Tot}\left({\sI}^{\bullet \bullet}\right)$ 
is a $K$-injective resolution of $f^*(E)$. Furthermore,
$f^*\left({\tau}^{\ge n} E\right) \to {\tau}^{\ge n} 
\left({\sI}^{\bullet \bullet}\right)$ is a Cartan-Eilenberg resolution, and 
hence $f^*\left({\tau}^{\ge n} E\right) \to {\sI}^{[\ge n]}
= {Tot}\left({\tau}^{\ge n} \left({\sI}^{\bullet \bullet}\right)\right)$
is a $K$-injective resolution for each $n$. One also checks that
${\sI} \xrightarrow {\cong} ~{\underset {n}{\varprojlim}}~ {\sI}^{[\ge n]}$.  
In particular, we get 
\begin{equation}\label{eqn:ex12*}
Rf_* \circ f^*(E) \cong f_*({\sI}) \cong f_*\left(
~{\underset {n}{\varprojlim}}~ {\sI}^{[\ge n]}\right).
\end{equation}
We claim that the natural map $f_*\left(
~{\underset {n}{\varprojlim}}~ {\sI}^{[\ge n]}\right) \to
~{\underset {n}{\varprojlim}}~ f_*\left({\sI}^{[\ge n]}\right)$ is an
isomorphism. To see this, we note first from Lemma~\ref{lem:injective}
that $Mod(X)$ and $Mod(Y)$ are Grothendieck categories. 
Hence we can use \cite[Lemma~2.3]{Krishna} to get
\[
{\left(f_*\left(~{\underset {n}{\varprojlim}}~ 
{\sI}^{[\ge n]}\right)\right)}^i = 
f_*\left({\left({~{\underset {n}{\varprojlim}}~} 
{\sI}^{[\ge n]}\right)}^i\right) = 
f_*\left(~{\underset {n}{\varprojlim}}~ {{\sI}^{[\ge n]}}^i\right).\]
On the other hand, since $f_*$ has a left adjoint, we have
\[
f_*\left(~{\underset {n}{\varprojlim}}~ {{\sI}^{[\ge n]}}^i\right) =
~{\underset {n}{\varprojlim}}~ f_*\left({{\sI}^{[\ge n]}}^i\right) =
~{\underset {n}{\varprojlim}}~ 
{\left(f_*\left({\sI}^{[\ge n]}\right)\right)}^i =
{\left(~{\underset {n}{\varprojlim}}~ 
f_*\left({\sI}^{[\ge n]}\right)\right)}^i,\]
where the last isomorphism follows again from \cite[Lemma~2.3]{Krishna}. 
This proves the claim. 

Using this claim in ~\eqref{eqn:ex12*}, we obtain
\begin{equation}\label{eqn:ex13*}
Rf_* \circ f^*(E) \cong 
~{\underset {n}{\varprojlim}}~ f_*\left({\sI}^{[\ge n]}\right)  \cong 
~{\underset {n}{\varprojlim}}~ Rf_* \circ f^*\left({\tau}^{\ge n}(E)\right).
\end{equation}
Finally, we consider the following commutative diagram.
\[
\xymatrix{
E \ar[r] \ar[d]_{\cong} & Rf_* \circ f^*(E) \ar[d] \\
~{\underset {n}{\varprojlim}}~ {\tau}^{\ge n}(E) \ar[r] &
~{\underset {n}{\varprojlim}}~ Rf_* \circ f^*\left({\tau}^{\ge n}(E)\right)}
\]
The left vertical arrow is an isomorphism by ~\eqref{eqn:ex11*} and the right
vertical arrow is an isomorphism by ~\eqref{eqn:ex13*}. Since 
${\tau}^{\ge n}(E)$ is a good truncation of $E$, we see that 
${\tau}^{\ge n}(E) \in D^{+}_{qc}(X \ {\rm on} \ Z)$ for each $n$ and 
hence the bottom horizontal arrow is an isomorphism by 
Proposition~\ref{prop:EXC1}.
We now conclude that $E \to  Rf_* \circ f^*(E)$ is an isomorphism.
\end{proof}
\begin{thm}\label{thm:EXCM}
Let $f : Y \to X$ be a representable and \'etale morphism of stacks.
Let $Z \overset{i}{\inj} X$ be a closed
substack such that the restriction $f : Z {\times}_X Y \to Z$ is an
isomorphism. Then the pullback $f^*: D_{qc}(X \ {\rm on} \ Z) \to
D_{qc}(Y \ {\rm on} \ Z{\times}_X Y)$ is an equivalence of triangulated
categories. 
\end{thm}    
\begin{proof}
Before we begin the proof, one should observe that $f^*$ maps
$D_{qc}(X \ {\rm on} \ Z)$ into $D_{qc}(Y \ {\rm on} \ Z{\times}_X Y)$ 
because $f$ is \'etale. Put $W = Z {\times}_X Y$. We first show that $f^*$ 
induces a full embedding of $D_{qc}(X \ {\rm on} \ Z)$ into 
$D_{qc}(Y \ {\rm on} \ W)$. In fact, we will show that the former is a 
reflexive full subcategory of the latter.

To show this, let $E, E' \in D_{qc}(X \ {\rm on} \ Z)$. We have then
\[
\begin{array}{llll}
{\rm Hom}_{D_{qc}(X \ {\rm on} \ Z)} \left(E, E'\right) & \cong & 
{\rm Hom}_{D_Z(qc/X)} \left(E, Rf_* \circ f^*(E')\right) &
({\rm by \ Lemma~\ref{lem:EXC2}}) \\
& \cong &  {\rm Hom}_{D_{qc}(X)} \left(E, Rf_* \circ f^*(E')\right) & \\
& \cong & {\rm Hom}_{D_{qc}(Y)} \left(f^*(E), f^*(E')\right) & \\
& \cong & {\rm Hom}_{D_{qc}(Y \ {\rm on} \ W)} \left(f^*(E), f^*(E')\right), &
\end{array}
\] 
where the third isomorphism follows from \cite[Lemma~3.3]{Krishna}.
Hence $f^*$ is full and faithful. Lemma~\ref{lem:EXC2} now implies 
that $D_{qc}(X \ {\rm on} \ Z)$ is a reflexive full triangulated subcategory of
$D_{qc}(Y \ {\rm on} \ W)$. To prove the theorem, it suffices now to show that
$f^*$ is surjective on objects. 

To prove this, let $F \in D_{qc}(Y \ {\rm on} \ W)$.
We first consider the case when $F$ is bounded above. If $X$ is a scheme,
then we have seen before that $Y$ is also a scheme and the assertion then
follows directly from the stronger result in \cite[Theorem~2.6.3]{TT}
that the map $f^*Rf_*(F) \to F$ is an isomorphism. If $X$ is a stack,
we can choose an atlas $U \xrightarrow{u} X$ and note that $U {\times}_X Y$
is a scheme since $f$ is representable and \'etale and hence strongly
representable. Now we argue as in the proof of Proposition~\ref{prop:EXC1}
to show that $f^*Rf_*(F) \to F$ is an isomorphism.  It follows moreover
from Corollary~\ref{cor:K-injAD*} that $Rf_*(F) \in D_{qc}(X \ {\rm on} \ Z)$.

To complete the proof in general, we write 
\begin{equation}\label{eqn:ex24}
F = ~{\underset {n}{\varinjlim}}~ {\tau}^{\le n}(F).
\end{equation}
We get natural maps
\[
f^*\left(~{\underset {n}{\varinjlim}}~ 
Rf_*\left({\tau}^{\le n}(F)\right)\right) 
\cong ~{\underset {n}{\varinjlim}}~
f^* \circ Rf_*\left({\tau}^{\le n}(F)\right) \to
~{\underset {n}{\varinjlim}}~ {\tau}^{\le n}(F) \to F,
\]
where the first isomorphism follows from ~\eqref{eqn:ex03*}.
Since each ${\tau}^{\le n}(F) \in D^{-}_{qc}(Y \ {\rm on} \ W)$, we have just 
shown that the second map is also an isomorphism, and the last map is an 
isomorphism by ~\ref{eqn:ex24}. \\
Putting $E = ~{\underset {n}{\varinjlim}}~ 
Rf_*\left({\tau}^{\le n}(F)\right)$, we see that $f^*(E) \cong F$. Thus,
it suffices to prove that $E \in D_{qc}(X \ {\rm on} \ Z)$ to finish the proof 
of the theorem. 

Let $T = X - Z$ and let $j: T \inj X$ be the inclusion morphism.
It suffices to show that $j^*(E) = 0$. To show this,
we have for any $i \in \Z$,
\[
\begin{array}{lll}
H^i\left(j^*(E)\right) & \cong & j^*\left(H^i(E)\right) \\
& \cong & j^*\left(H^i\left(~{\underset {n}{\varinjlim}}~ 
Rf_*\left({\tau}^{\le n}(F)\right)\right)\right) \\
& \cong & ~{\underset {n}{\varinjlim}}~ j^* \circ H^i 
\left(Rf_*\left({\tau}^{\le n}(F)\right)\right) \\
& \cong & ~{\underset {n}{\varinjlim}}~ 
H^i\left(j^* \circ Rf_*\left({\tau}^{\le n}(F)\right)\right) \\
& \cong & 0,
\end{array}
\]
where the last term is zero since ${\tau}^{\le n}(F) \in 
D^{-}_{qc}(Y \ {\rm on} \ W)$ and hence 
$Rf_*\left({\tau}^{\le n}(F)\right) \in D^{-}_{qc}(X \ {\rm on} \ Z)$
as shown above. This completes the proof of the theorem.
\end{proof}
\begin{thm}[\bf{Excision}]\label{thm:Excision}
Let $f : Y \to X$ be a representable and \'etale morphism of tame stacks which
admit coarse moduli schemes. Let $Z \overset{i}{\inj} X$ be a closed
substack such that the restriction $f : Z {\times}_X Y \to Z$ is an
isomorphism. Then the map $f^*: C(X) \to C(Y)$ induces an equivalence of
spectra $K(X \ {\rm on} \ Z) \xrightarrow{\cong} 
K(Y \ {\rm on} \ Z{\times}_XY)$.
\end{thm}
\begin{proof} If we let $W = Z {\times}_X Y$, then $f^*$ clearly induces
the morphism of dg-categories $f^*: C_Z(Perf/X) \to C_W(Perf/Y)$. To show
that the induced morphism of the $K$-theory spectra is a homotopy equivalence,
it suffices to show that the map $f^*:D_Z(Perf/X) \to D_W(Perf/Y)$ is
an equivalence of the triangulated categories. Suppose we show that
\begin{equation}\label{eqn:Excision1}
f^*:D_{qc}(X \ {\rm on} \ Z) \to D_{qc}(Y \ {\rm on} \ W)
\end{equation}
 is an equivalence. Since $f$ is \'etale, $f^*$ takes perfect complexes
which are acyclic on the complement of $Z$ to perfect complexes on $Y$
which are acyclic on the complement of $W$. It follows from
Theorem~\ref{thm:Perf-comp} that $f^*$ takes compact objects of
$D_{qc}(X \ {\rm on} \ Z)$ into the compact objects of
$D_{qc}(Y \ {\rm on} \ W)$.
The equivalence in ~\eqref{eqn:Excision1} will then restrict to an equivalence
of the corresponding full subcategories of compact objects. Another
application of Theorem~\ref{thm:Perf-comp} will show that $f^*$ induces
an equivalence between the derived categories of perfect complexes.
Thus, we only need to show ~\eqref{eqn:Excision1}. But this is shown
in Theorem~\ref{thm:EXCM}.
\end{proof} 
\begin{cor}[\bf{Mayer-Vietoris}]\label{cor:MV}
Let $f : Y \to X$ be a representable and \'etale morphism of tame stacks which
admit coarse moduli schemes. Let 
\begin{equation}\label{eqn:MV1}
\xymatrix@C.8pc{
V \ar[r] \ar[d]_{g} & Y \ar[d]^{f} \\
U \ar[r] & X}
\end{equation}
be a Cartesian square such that $U \inj X$ is an open substack and $f$ is an
isomorphism over the complement of $U$. Then, this induces a homotopy
Cartesian square of $K$-theory spectra
\begin{equation}\label{eqn:MV2}
\xymatrix@C.8pc{
K(X) \ar[r] \ar[d]_{f^*} & K(U) \ar[d]^{g^*} \\
K(Y) \ar[r] & K(V).}
\end{equation}
\end{cor}
\begin{proof}
Let $Z \overset{i}{\inj} X$ be the complement of $U$ and let
$W = Z {\times}_X Y$. By Theorem~\ref{thm:Localization}, there is a 
commutative diagram of morphism of spectra
\begin{equation}\label{eqn:MV3}
\xymatrix@C.6pc{
K(X \ {\rm on} \ Z) \ar[r] \ar[d] & K(X) \ar[r] \ar[d]_{f^*} & 
K(U) \ar[d]^{g^*} \\
K(Y \ {\rm on} \ W) \ar[r] & K(Y) \ar[r] & K(V)}
\end{equation} 
such that each row is a homotopy fibration. Moreover, the left vertical 
morphism is a homotopy equivalence by Theorem~\ref{thm:Excision}. This proves 
the corollary.
\end{proof}
\section{Nisnevich site of stacks}\label{section:NIS}
In this section, we define the Nisnevich site of a stack and study its basic
properties. It will turn out in the subsequent sections that the 
resulting Grothendieck topology is given by a certain 
$cd$-structure in the sense of \cite{Voev1} on the big \'etale site of $X$. 
This fact will be subsequently used in this paper to prove the Nisnevich
descent for the $K$-theory of perfect complexes on stacks. 
The equivariant Nisnevich site in the category $Sm^G_k$ of
smooth schemes over a field with action of a finite group, has also appeared
in \cite{HKO} in the study of equivariant stable homotopy theory.

Let $S$ be a 
noetherian affine scheme. Fron now on, we shall assume that all the objects in 
the category ${\sD}{\sM}_S$ are of finite type over $S$. Recall that a 
Deligne-Mumford stack $X$ is of finite type over $S$ if there is an 
\'etale atlas $U \xrightarrow{u} X$ such that $U$ is a scheme of finite type 
over $S$. Recall from Remark~\ref{remk:topos} that $\sE(X)$ is the category
whose objects are morphisms $[X' \xrightarrow{f} X]$ in
${\sD}{\sM}_S$, where $f$ is \'etale but not necessarily representable.

\subsection{Nisnevich topology}\label{subsection:NIS*}
Let $X$ be a stack. We recall from \cite[Chapter~5]{LMB} that a point
of the stack $X$ is an equivalence class of morphisms 
${\rm Spec}(K) \xrightarrow{x} X$, where $K$ is a $S$-field and where
the two points $(x, K)$ and $(y, L)$ are equivalent if there is
a common field extension $E$ of $K$ and $L$ such that the composite maps
${\rm Spec}(E) \to  {\rm Spec}(K) \xrightarrow{x} X$ and
${\rm Spec}(E) \to  {\rm Spec}(L) \xrightarrow{y} X$ are isomorphic.
If $X$ is a scheme, this gives the usual notion of points on schemes.
We often denote a point $(x, K)$ of $X$ simply by $x$ keeping the above
meaning in mind. Every point $x : {\rm Spec}(K) \to X$ has a unique
factorization 
\[
{\rm Spec}(K) \overset{\ov{x}}{\surj} \eta_x \inj X
\]
as $fpqc$ stacks, and it turns out that $\eta_x$ is in fact a $fppf$ stack
and hence a Deligne-Mumford stack. Moreover, there is a uniquely defined
field $k(\eta_x)$ with a unique map 
$\eta_x \to {\rm Spec}\left(k(\eta_x)\right)$ which is a gerbe. The field
$k(\eta_x)$ is called the {\sl residue field} of $x$ (or $\eta_x$) and
$\eta_x$ is called the {\rm residual gerbe} at $x$. Moreover, one can always
choose a representative $(x, K)$ of $x$ such that $k(\eta_x) \inj K$ is a
finite extension and there is a Cartesian diagram of stacks
\begin{equation}\label{eqn:point}
\xymatrix@C.7pc{
[{\rm Spec}(K)/{G_{\eta}}] \ar[r] \ar[d] & \eta_x \ar[d] \\
{\rm Spec}(K) \ar[r] & {\rm Spec}\left(k(\eta_x)\right),}
\end{equation}
where $G_{\eta}$ is the isotropy group scheme at the point $x$. The residual
gerbe $\eta_x$ is uniquely defined by the equivalence class of $x$ and 
$G_{\eta}$ is a group scheme over $\eta_x$. Note that $G_{\eta}$ is
a finite \'etale group scheme since $X$ is a Deligne-Mumford stack.
We shall often denote the residual gerbe $\eta_x$ at the point $x$ simply
by $\eta$. The morphism $\eta_x \to X$ is a monomorphism and representable
and the map $\eta_x \to {\rm Spec}\left(k(\eta_x)\right)$ is of finite type.
It follows in particular that this map is a coarse moduli space map.
Every stack is a union of its finitely many irreducible components and every
irreducible component has a unique generic point. The residual gerbe at
the generic point of an irreducible component will be called a generic
residual gerbe of the stack. The following elementary results about the 
residual gerbes on Deligne-Mumford stacks are well known.
\begin{lem}\label{lem:gerbe}
Let $X$ be a stack with a coarse moduli space $\underline{X}$. Let $x$ be 
a point of $X$ and let ${\rm Spec}(k) = \underline{x} \in \underline{X}$ be 
its image in the coarse moduli space. Then $k$ is the residue field of
the residual gerbe $\eta$ at $x$ and the diagram 
\begin{equation}\label{eqn:point*}
\xymatrix@C.8pc{
\eta \ar[r] \ar[d] & X \ar[d]^{p} \\
\underline{x} \ar[r] & \underline{X}}
\end{equation}
is Cartesian.
\end{lem}
\begin{proof}(Sketch) The composite map ${\rm Spec}(K) \to X \to \underline{X}$
gives a commutative diagram
\begin{equation}\label{eqn:point*}
\xymatrix@C.8pc{
{\rm Spec}(K) \ar[r] \ar[drr] & \eta \ar[r] \ar[dr] & 
p^{-1}(\underline{x}) \ar[r] \ar[d]_{q} & X \ar[d]^{p} \\
& & \underline{x} \ar[r] & \underline{X},}
\end{equation}
where the square is Cartesian. Since $\underline{x}$ is the image of
a point on $X$, the
map $p^{-1}(\underline{x}) \to \underline{x}$ is a coarse moduli space map.
Since $\underline{x}$ is the spectrum of a field, it follows from
\cite[Theorem~11.5]{LMB}) that $q$ is a gerbe. It follows now from
[{\sl loc. cit.} Lemme~3.17] that $\eta \to p^{-1}(\underline{x})$ is an
epimorphism. Since $\eta \to X$ is a monomorphism, we see that
$\eta \to p^{-1}(\underline{x})$ is an isomorphism.
\end{proof}
\begin{lem}\label{lem:gerbeiso}
Let $f : X \to Y$ be a morphism of stacks. Let $x$ be a point of $X$ and let
$y$ be its image in $Y$. Then the induced map of the residual gerbes
$\eta_x \to \eta_y$ is an isomorphism if and only if the induced maps on 
the corresponding isotropy groups and the residue fields are isomorphisms.
\end{lem}
\begin{proof}
Since the spectrum of the residue field of the residual gerbe is its coarse
moduli space and since the residual gerbe has only one point, we see
easily that the isomorphism of the residual gerbes implies the isomorphism
of the residue fields and the isotropy groups. To prove the converse, let
$k$ be the residue field of $\eta_x$ and $\eta_y$. Then, there is a finite
extension $l$ of $k$ such that the base extension of $\eta_x$ and
$\eta_y$ over $l$ are the neutral gerbes of the form 
$[{{\rm Spec}(l)}/{G}]$ and $[{{\rm Spec}(l)}/{G}]$, where $G$ is the common
isotropy group. Thus, the natural map $\eta_x \to \eta_y$ becomes an
isomorphism after an \'etale cover. Hence the map must be an isomorphism.
\end{proof} 

\begin{defn}\label{defn:Nis-cite}
Let $X$ be a stack over $S$. A family of morphisms $\{X_i \xrightarrow{f_i} 
X\}$ in the \'etale site ${\sE}t(X)$ ({\sl cf.} Remark~\ref{remk:topos}) 
is called a {\sl Nisnevich} cover of $X$ if for every point $x \in X$, there 
is a member $X_i$ and a point $x_i \in X_i$ such that $f_i(x_i) = x$ and the 
induced map of the residual gerbes $\eta_{x_i} \to \eta_x$ is an isomorphism.
\end{defn}
\begin{lem}\label{lem:Nis-cite*}
The category ${\sD}{\sM}_S$ with covering maps 
$\{X_i \xrightarrow{f_i} X\}$ given by the
Nisnevich covers, defines a Grothendieck topology on ${\sD}{\sM}_S$.
\end{lem}
\begin{proof} All the conditions for a Grothendieck topology are immediate
except possibly the condition that a Nisnevich cover has the base change
property. To see this, let $Y \xrightarrow{f} X$ be a morphism in
${\sD}{\sM}_S$ and let $\{X_i \xrightarrow{f_i} X\}$ be a Nisnevich cover. 
It is clear that $X_i {\times}_X Y \to Y$ is \'etale. Now let $y \in Y$ be a 
point with $f(y) = x$ and let $x_i \in X_i$ be such that $f_i(x_i) = x$
and $\eta_{x_i} \xrightarrow{\cong} \eta_x$. This implies that $\eta_{x_i}
{\times}_{\eta_x} \eta_y \xrightarrow{\cong} \eta_y$. Thus we see that
$\eta_{x_i}{\times}_{\eta_x} \eta_y$ defines a point $y_i \in X_i{\times}_X Y$
which maps to $y$ such that $\eta_{y_i} \xrightarrow{\cong} \eta_y$.
\end{proof}    
We shall call the Grothendieck topology defined by the Nisnevich coverings,
the {\sl Nisnevich topology} on ${\sD}{\sM}_S$. The corresponding site
will be denoted by $\left({\sD}{\sM}_S\right)_{Nis}$. Note that since all
stacks in this category are noetherian and of finite type over $S$,
every Nisnevich cover has a refinement by a finite cover. In other words,
$\left({\sD}{\sM}_S\right)_{Nis}$ is a noetherian site.

Let ${\left({\sD}{\sM}_S\right)}^r_{Nis}$ be the topology on ${\sD}{\sM}_S$
generated by those Nisnevich coverings $\{X_i \xrightarrow{f_i} X\}$
where each $f_i: X_i \to X$ is representable. 
Since the representable morphisms have the base change property, we see that
${\left({\sD}{\sM}_S\right)}^r_{Nis}$ is also a Grothendieck site on
${\sD}{\sM}_S$. 

\subsection{Stacks with moduli schemes}\label{subsection:St-mod}
To prove the main results of this paper, we mostly focus on 
those stacks which have coarse moduli schemes. We need the following results
in order to show that the Nisnevich topology restricts to a similar topology
on the category of those stacks which have the coarse moduli schemes.
We first quote the following result from \cite{DR}.
\begin{thm}\label{thm:mod-sp}
Let $X$ be stack over $S$. Then $X$ has the coarse moduli space 
$X \xrightarrow{p} \un{X}$, where $\un{X}$ is an algebraic space. The map $p$ 
is separated, proper and quasi-finite. Moreover, $\un{X}$ is separated, 
noetherian and of finite type over $S$. In particular, $\un{X} \in
{\sD}{\sM}_S$.
\end{thm}
\begin{proof}
{\sl Cf.} \cite[Theorem~6.12]{DR}.
\end{proof}

\begin{lem}\label{lem:mod-sp*}
Let 
\begin{equation}\label{eqn:mod-sp*1}
\xymatrix@C.8pc{
X \ar[r]^{f} \ar[d]_{p} & Y \ar[d]^{q} \\
\un{X} \ar[r]_{g} & \un{Y}}
\end{equation}
be a commutative diagram in ${\sD}{\sM}_S$ ({\sl cf.} Theorem~\ref{thm:mod-sp})
where the bottom arrow is the map of coarse moduli spaces of the top arrow. 
Then, \\
$(i)$ $g$ is separated, quasi-compact and of finite type. \\
$(ii)$ $f$ is quasi-finite $\Rightarrow$ $g$ is quasi-finite. \\
$(iii)$ $f$ is closed $\Rightarrow$ $g$ is closed. \\
$(iv)$ $f$ is proper and $Y$ tame $\Rightarrow$ $g$ is proper. \\
$(v)$  $f$ quasi-finite and $\un{Y}$ is a scheme $\Rightarrow$ $\un{X}$
is a scheme. \\
\end{lem}
\begin{proof}
Since the moduli spaces are quasi-compact (in fact noetherian), separated 
and of finite type over $S$ by Theorem~\ref{thm:mod-sp}, we conclude from 
Proposition~\ref{prop:QCS} that $g$ is quasi-compact and separated and of
finite type. If $f$ is quasi-finite, then Theorem~\ref{thm:mod-sp} implies
that $g \circ p$ is quasi-finite. But, $g$ must then be quasi-finite.

Suppose now that $f$ is closed and let $\un{Z} \subset \un{X}$ be a 
closed subspace. Then $Z = p^{-1}(\un{Z}) = \un{Z} {\times}_{\un{X}} X$
is a closed substack of $X$ and hence $g(\un{Z}) = g \circ p(Z) = q \circ f(Z)$
is closed in $\un{Y}$ since $q$ is proper by Theorem~\ref{thm:mod-sp} and
hence closed. 

Assume now that $Y$ is a tame stack and $f$ is proper. We need to
show that $g$ is universally closed. So let $\un{Z} \xrightarrow{h} \un{Y}$
be a morphism of algebraic spaces. Since $Y$ is tame, it follows from
\cite[Lemma~2.3.3]{Ab-V} that $\un{Z} {\times}_{\un{Y}} Y \to \un{Z}$ is a
coarse moduli space. We need to show that this map is closed. Since the base
change of a proper map is proper and hence closed, the above proof now
shows that the map $\un{Z} {\times}_{\un{Y}} Y \to \un{Z}$ is closed.

Finally, assume that $\un{Y}$ is a scheme and $f$ is quasi-finite. Then we
see from the other assertions of the lemma that $g$ is a quasi-finite map
of noetherian algebraic spaces which is of finite type and separated with
$\un{Y}$ a scheme. Hence, it follows from \cite[Corollary~II.6.16]{Knutson}
that $\un{X}$ is also a scheme.
\end{proof}
\begin{prop}\label{prop:mod-sch}
Let $S$ be the spectrum of a field and let $\wt{{\sD}{\sM}}_S$ be the
full subcategory of ${\sD}{\sM}_S$ consisting of those stacks which have
coarse moduli schemes. Then $\wt{{\sD}{\sM}}_S$ is closed under fiber
products.
\end{prop}
\begin{proof}
Let $X \xrightarrow{f} Z$ and $Y \xrightarrow{g} Z$ be the morphisms 
in ${\sD}{\sM}_S$ and let $W = X {\times}_Z Y$. We need to show that
$W$ has a coarse moduli scheme if $X$, $Y$ and $Z$ have so. We first 
notice that $W$ has a coarse moduli space by Theorem~\ref{thm:mod-sp}.
We next observe that the diagram 
\[
\xymatrix{
W \ar[r]^{h} \ar[d] & X {\times}_S Y \ar[d] \\
Z \ar[r]_{\Delta_Z} & Z {\times}_S Z}
\]
is a Cartesian diagram. Since $Z$ is a Deligne-Mumford stack, $\Delta_Z$
is separated and quasi-finite by \cite[Lemme~4.2]{LMB} and hence so is the
map $h$. We see from Lemma~\ref{lem:mod-sp*} that the coarse moduli space
of $W$ is a scheme if the same holds for the coarse moduli space of
$X {\times}_S Y$. 

Let us now assume that the coarse moduli spaces $\un{X}$ and $\un{Y}$ of $X$ 
and $Y$ respectively are schemes.
Since $S$ is the spectrum of a field, the map $\un{X} {\times}_S \un{Y} 
\to \un{Y}$ is flat and hence the map $\un{X} {\times}_S Y \to 
\un{X} {\times}_S \un{Y}$ is a coarse moduli space by \cite[Lemma~2.2.2]{Ab-V}.
On the other hand, the map $X {\times}_S Y \to \un{X} {\times}_S Y$
is of finite-type, separated, and quasi-finite map of noetherian stacks
by Theorem~\ref{thm:mod-sp}. Hence the map $T \to \un{X} {\times}_S \un{Y}$
is of finite type, separated and quasi-finite map of noetherian algebraic
spaces by Lemma~\ref{lem:mod-sp*}, where $T$ is the coarse moduli space
of $X {\times}_S Y$. It follows again from Lemma~\ref{lem:mod-sp*} that $T$
is a scheme.
\end{proof}
\begin{remk}\label{remk:tame-fiber-prod}
The reader would notice in the above proof that the only reason to assume $S$ 
to the spectrum of a field was to ensure that the moduli space has the flat 
base change. However, this base change also holds if all the stacks considered 
are tame by \cite[Lemma~2.3.3]{Ab-V}. Thus we see that if we are working in the
category of tame stacks, then Proposition~\ref{prop:mod-sch} holds over
any noetherian and affine base $S$.
\end{remk}
\begin{cor}\label{cor:mod-sch*}
Let ${\left(\wt{{\sD}{\sM}}_S\right)}_{Nis}$ and 
${\left(\wt{{\sD}{\sM}}_S\right)}^r_{Nis}$ denote the restrictions of 
${\left({\sD}{\sM}_S\right)}_{Nis}$ and ${\left({\sD}{\sM}_S\right)}^r_{Nis}$ 
respectively to the full subcategory of those stacks which have coarse moduli 
schemes. Then ${\left(\wt{{\sD}{\sM}}_S\right)}_{Nis}$ 
and ${\left(\wt{{\sD}{\sM}}_S\right)}^r_{Nis}$ are Grothendieck sites.
\end{cor}
\begin{proof} Follows immediately from Proposition~\ref{prop:mod-sch}.
\end{proof}
\begin{remk}\label{remk:G-maps}
Let $G$ be an affine and smooth group scheme over $S$ and let $Sch^G_S$
denote the category of noetherian and separated schemes of finite type over
$S$ with $G$-action and $G$-equivariant morphisms. For a $G$-scheme
$X$, let $[X/G]$ denote the associated quotient stack. It follows essentially
from the definitions that the atlas $X \to [X/G]$ is strongly representable. 
Moreover, for a $G$-equivariant map $X \xrightarrow{f} Y$, the associated map 
of stacks $[X/G] \to [Y/G]$ is always representable. In particular, it is 
strongly representable if $f$ is \'etale. We conclude that if we
consider the $G$-schemes with proper action, then $Sch^G_S$ is canonically
equivalent to a subcategory (not full) $\wt{Sch}^G_S$ of ${\sD}{\sM}_S$
whose Nisnevich topology is in fact the restriction of the Nisnevich site
${\left({\sD}{\sM}_S\right)}^r_{Nis}$. 
 \end{remk}
\begin{exm}\label{exm:Nis-cov}
Let $G$ be a smooth affine group scheme over $S$ and let $Sch^G_S$ be as
in Remark~\ref{remk:G-maps}. The following example shows that for a 
$G$-equivariant map $X \xrightarrow{f} Y$ which is a Nisnevich cover of $Y$ in 
the category ${Sch}/S$ of schemes over $S$, the associated map of quotient 
stacks need not be a Nisnevich cover.   

Let $k$ be a field with $char(k) \neq 2$. Let $G = {\Z}/2$ acting trivially on
$Y = {\rm Spec}(k)$. Let $X = {\rm Spec}(k) \coprod {\rm Spec}(k)$ where
$G$ acts by switching the components. The natural morphism $f : X \to Y$ which
is identity on both components, is clearly $G$-equivariant. Moreover,
$f$ is obviously a Nisnevich covering map of schemes if we forget the group
action ({\sl cf.} \cite[Remark~4.1]{Serpe1}). However, the associated
map $\ov{f}: [X/G] \to [Y/G]$ in the category ${\sD}{\sM}_k$ is \'etale but 
not a Nisnevich cover. This is because the isotropy group of every point of 
$[X/G]$ is trivial while the isotropy group of the only point of $[Y/G]$ is all 
of $G$. Hence by Lemma~\ref{lem:gerbeiso}, $\ov{f}$ can not be a Nisnevich 
cover. In fact, it is not very difficult to see that a morphism $g :
W \to [k/G]$ from a stack is a Nisnevich cover if and only if this map is
\'etale and has a section.
\end{exm}

\section{Nisnevich topology via $cd$-structure}\label{section:cdS}
In \cite{Voev1}, Voevodsky introduced the notion of $cd$-structures on
a category $\sC$. It turns out that such a $cd$-structure naturally defines a 
Grothendieck topology on $\sC$. He also showed that a Grothendieck
topology which comes from a $cd$-structure, has several nice properties.
We refer the reader to \cite{Voev1} and \cite{Voev2} for applications in the
homotopy theory of simplicial sheaves on schemes. It was also shown by
Voevodsky that the Nisnevich topology on schemes is induced by a 
$cd$-structure. In this section, we use these ideas of Voevodsky to define
a $cd$-structure on ${\sD}{\sM}_S$ and then show that the resulting 
Grothendieck topology coincides with the Nisnevich topology as defined 
in Section~\ref{section:NIS}. This will be later used to prove the Nisnevich
descent for the $K$-theory of perfect complexes on stacks.

Let $\sC$ be a category with an initial object. Recall from \cite{Voev1} 
that a $cd$-structure on $\sC$ is a collection $P$ of commutative squares of 
the form 
\begin{equation}\label{eqn:Dist-square}
\xymatrix@C.8pc{
B \ar[r] \ar[d] & Y \ar[d]^{p} \\  
A \ar[r]_{e} & X}
\end{equation}
such that if $Q \in P$ and $Q'$ is isomorphic to $Q$, then $Q' \in P$.
The squares of the collection $P$ are called the {\sl distinguished squares}
of the $cd$-structure. One can define different $cd$-structures and/or
restrict them to subcategories considering only the squares which lie in the
corresponding subcategory. In particular, for an object $X$ of $\sC$, 
the $cd$-structure of $\sC$ defines a $cd$-structure on the category  
${\sC}/X$.

We define the topology $t_P$ associated with a $cd$-structure $P$ as the
smallest Grothendieck topology such that for a distinguished square of
the form ~\eqref{eqn:Dist-square}, the morphisms $\{Y \xrightarrow{p} X,
A \xrightarrow{e} X\}$ form a covering and such that the empty sieve
is a covering of the initial object. 
The class of simple coverings in this topology is the smallest class of
families which contains all isomorphisms and satisfies the condition that
for a distinguished square of the form ~\eqref{eqn:Dist-square} and
families $\{p_i: Y_i \to Y\}$, $\{q_j :A_j \to A\}$ in this class, the family
$\{p \circ p_i, e \circ e_j\}$ is also a covering in this family. In other
words, a simple covering is an iteration of the coverings coming from 
distinguished squares.
 
We now define a $cd$-structure on ${\sD}{\sM}_S$. Note that this category
has an initial object in terms of the empty scheme.
\begin{defn}\label{defn:Nis-cd}
A square in ${\sD}{\sM}_S$ is called {\sl distinguished} if it is
a Cartesian diagram of the form ~\eqref{eqn:Dist-square},
where $e$ is an open immersion of stacks with complement $Z$ and $p$ is an 
\'etale morphism (not necessarily representable) of stacks such that the 
induced morphism $Z {\times}_X Y \to Z$ is an isomorphism of the associated
reduced closed substacks. 
\end{defn}
It is clear from the definition that these distinguished squares define a 
$cd$-structure on ${\sD}{\sM}_S$. Let $t_{Nis}$ denote the Grothendieck 
topology on ${\sD}{\sM}_S$ defined by this $cd$-structures. We shall call
this as the {\sl Nisnevich cd}-topology on ${\sD}{\sM}_S$.

Let $t^r_{Nis}$ denote the Grothendieck topology on ${\sD}{\sM}_S$ 
defined by the class of those distinguished squares of the form
~\eqref{eqn:Dist-square}, where the map $p$ is also assumed to be
representable. We shall often call the corresponding $cd$-structure
as the {\sl representable Nisnevich $cd$-structure}.
We denote the restriction of $t^r_{Nis}$ on the full subcategory
$\wt{{\sD}{\sM}}_S$ by $\wt{t}^r_{Nis}$.

It is clear from the above that every covering in the Nisnevich 
$cd$-topology is also a covering in the Nisnevich topology on
${\sD}{\sM}_S$. Our aim now is to show that both define the same topology
on ${\sD}{\sM}_S$. We follow the argument of Voevodsky 
\cite[Proposition~2.17]{Voev2} in order to prove this.

We recall that a morphism $f : Y \to X$ of stacks is called {\sl split} if
there is a sequence of closed embeddings 
\begin{equation}\label{eqn:s-sequence}
\emptyset = Z_{n+1} \to Z_n \to \cdots \to Z_1 \to Z_0 = X
\end{equation}
such that for each $i = 0, \cdots , n$, the morphism $(Z_i - Z_{i+1}) 
{\times}_X Y \to (Z_i - Z_{i+1})$ has a section.
\begin{lem}\label{lem:section}
Let $X$ be an irreducible stack with the generic residual gerbe $\eta_X$ 
and let $Y$ be an irreducible and noetherian $S$-scheme with the generic 
point $\eta_Y$. Let $t : \eta_Y \to \eta_X$ be a morphism. Then, there is an 
affine open subset $U \subset Y$ and morphism $\wt{t} : U \to X$ which 
restricts to $t$ on $\eta_Y$.
\end{lem}
\begin{proof} We can assume that $Y = {\rm Spec}(A)$ is affine. Let
$(x, K)$ be a representative of the generic point of $X$ with the residual
gerbe $\eta_X$. Then by \cite[Th\'eo\'rem~6.2]{LMB}, there is an irreducible
affine and noetherian scheme $W$ with an action of a finite group $G$ and a 
Cartesian diagram
\begin{equation}\label{eqn:section1}
\xymatrix@C.9pc{
{\rm Spec}(K) \ar[r] \ar[d]_{id} & [W/G] \ar[d]^{\phi} \\
{\rm Spec}(K) \ar[r] & X,}
\end{equation}
where $\phi$ is representable, \'etale and separated. Thus, we can assume
that our stack $X$ is of the form $[W/G]$, where $G$ is a finite group 
acting on an irreducible, affine and noetherian scheme $W$.

Since a map $Z \to X = [W/G]$ is equivalent to a diagram
\begin{equation}\label{eqn:section2}
\xymatrix@C.9pc{
Z' \ar[r]^{f} \ar[d]_{\pi} & W \ar[d]^{p} \\
Z \ar[r] & X,}
\end{equation}
where $\pi$ is a principal $G$-bundle and $f$ is a $G$-equivariant map,
we see that the morphism $t : \eta_Y \to \eta_X \inj X$ is equivalent to
a principal $G$-bundle $T \xrightarrow{\pi} \eta_Y$ and a $G$-equivariant
map $T \xrightarrow{f} W$. Since $G$ is a finite group, the principal 
$G$-bundle $\pi$ extends to a principal $G$ bundle 
$V' \xrightarrow{\wt{\pi}} V$,
where $V \subset Y$ is an affine open subset. Thus, we can replace $Y$ by
the appropriate open set and get a diagram of the fom

\begin{equation}\label{eqn:section3}
\xymatrix@C.9pc{
Y' \ar[d]_{\wt{\pi}} & T \ar[r]^{f} \ar[l] \ar[d]^{\pi} & W \ar[d]^{w} \\
Y & \eta_Y \ar[l] \ar[r]_{t} & X,}
\end{equation}
where all the vertical maps are $G$-bundles and the top horizontal map
is $G$-equivariant. Since $T$ is the generic fiber of $\wt{\pi}$, it is
the set of generic points of $Y'$ and hence is a finite set. In particular,
there is an open subset $V' \subset Y'$ containing $T$ and a morphism
(not necessarily $G$-equivariant) $f' : V' \to W$ which extends the map 
$f$.  Put 
\[
\wt{T} = {\underset{g \in G}\bigcap} gV' \subset V' \ {\rm and \ let} \ \ 
 \wt{f} = {f'}|_{\wt{T}}.
\]
Since $V'$ contains all the generic points of $Y'$, it is dense and hence
$gV'$ is dense open in $Y'$ for all $g \in G$. This implies that $\wt{T}$
is open dense in $Y'$ which is now $G$-invariant. Letting $U = {\wt{T}}/G$,
we see that the map $t$ extends to a map $\wt{t}: U \to X$ where 
$U$ is an open subscheme of $Y$. By shrinking, we can assume further that
$U$ is affine.  
\end{proof}
\begin{lem}\label{lem:section*}
Let $f : Y \to X$ be an \'etale morphism of irreducible stacks which is an
isomorphism at the generic residual gerbes. Then there are dense open substacks
$V \inj Y$ and $U \inj X$ such that $g = f|_{V} : V \to U$ is an isomorphism.
\end{lem}
\begin{proof} Since $X$ is a separated Deligne-Mumford stack which is 
irreducible, we can use \cite[Th\'eor\'eme~6.1]{LMB} and assume that 
$X = [W/G]$, where $G$ is a finite group acting on an irreducible affine
and noetherian scheme $W$. Thus we have a commutative diagram
\begin{equation}\label{eqn:section*3}
\xymatrix@C.6pc{
\wt{Y} \ar[r]^{p} \ar[d]_{g} & Y \ar[d]^{f} & \eta_Y \ar[l] \ar[d]^{\cong} \\
W \ar[r]_{q} & X & \eta_X \ar[l]}
\end{equation}
where $\eta_X$ and $\eta_Y$ are the generic residual gerbes of $X$ and $Y$
respectively and the square on the left is Cartesian. In particular,
$\wt{Y}$ is a noetherian and separated Deligne-Mumford stack with an action
of $G$ such that $g$ is $G$-equivariant and the horizontal maps on the
left square are principal $G$-bundles. Let $\zeta$ denote the generic point
of $W$ and let $\wt{Y}_{\zeta}$ be the generic fiber of $g$. Since $\zeta$ is
the only point which maps onto $\eta_X$, the isomorphism of the right
vertical vertical map implies that the map $\wt{Y}_{\zeta} \to \zeta$
is an isomorphism. This in turn implies that $\wt{Y}$ has only one
generic point with the residual gerbe $\theta = \wt{Y}_{\zeta}$.

By Lemma~\ref{lem:section}, there is an affine open subset $\wt{V} \subset W$
and a map $h': \wt{V} \to \wt{Y}$ such that the composite $\wt{V} \to 
g^{-1}(\wt{V}) \to \wt{V}$
is identity as the composite is identity on the generic points.
Putting $\wt{U} =  {\underset{g \in G}\cap} g\wt{V}$, we see that 
$\wt{U} \subset W$ is $G$-invariant and there is a $G$-equivariant map 
$h: \wt{U} \to \wt{Y}$ such that the composite 
$\wt{U} \to g^{-1}(\wt{U}) \to \wt{U}$ is identity. Since $g$ is \'etale and 
the composite $g \circ h$ is identity, we 
see that $\wt{U}$ maps isomorphically onto an open substack $\wt{V'}$
of $\wt{Y}$ which is $G$-invariant. Putting $V = [{\wt{V'}}/G]$ and
$U = [{\wt{U}}/G]$, we see that $V$ is an open dense substack of $Y$ which
maps isomorphically to an open dense substack $U$ of $X$. 
\end{proof}
\begin{prop}\label{prop:section-M}
Let $f :Y \to X$ be an \'etale morphism of stacks. Assume that for every 
generic residual gerbe $\eta$ of $X$, there is a generic residual gerbe 
$\zeta$ of $Y$ which maps isomorphically onto $\eta$. Then there is an open 
substack $V$ of $Y$ containing a generic residual gerbe which maps 
isomorphically onto an open substack $U$ of $X$ containing a generic residual
gerbe.
\end{prop}
\begin{proof} We fix a generic gerbe $\eta$ of $X$ and let $Z \inj X$ be the
irreducible component of $X$ such that $\eta$ is the residual gerbe of the
generic point of $Z$. Let $Z' \inj X$ be the union of all other irreducible
components of $X$ and put $U = X-Z'$. Then $U$ is an open substack of $X$
which contains the generic residual gerbe $\eta$. We choose an irreducible
component $W$ of $Y$ such that the residual gerbe $\zeta$ at the generic point 
of $W$ maps isomorphically onto $\eta$ and let $W'$ be the union of all other
irreducible components of $Y$. Put $V = U {\times}_X (Y - W')$.

We now see that $V$ is an open and irreducible substack of $Y$ containing
the generic residual gerbe $\zeta$ such that $V$ maps to an open and 
irreducible substack $U$ of $X$ containing the generic residual gerbe $\eta$
of $X$. Hence by Lemma~\ref{lem:section*}, there are open substacks $V'
\inj V$ and $U' \inj U$ containing the generic residual gerbes $\zeta$ and
$\eta$ respectively such that $V'$ maps isomorphically onto $U'$. Since
$V'$ and $U'$ are also open substacks of $Y$ and $X$ respectively, we 
get the desired assertion.
\end{proof}
\begin{thm}\label{thm:split-M}
Let $X$ be a stack and let $\{X_i \xrightarrow{f_i} X\}$ be a finite Nisnevich
covering of $X$. Then the morphism $f = \coprod f_i : Y = \coprod X_i \to X$ is
split.
\end{thm}
\begin{proof} We choose a generic residual gerbe $\eta$ of $X$. By the 
definition of the Nisnevich covering, there a member $X_i$ of the family
and a point $x_i$ with the residual gerbe $\zeta$ such that $\zeta$ maps
isomorphically onto $\eta$. Since $f$ is \'etale, $x_i$ must be a generic
point of $X_i$ and hence of $Y$. By Proposition~\ref{prop:section-M},
there is an open substack $V$ of $Y$ containing $\zeta$ which maps 
isomorphically onto an open substack $U$ of $X$ containing $\eta$.
In particular, the map $U {\times}_X Y \to U$ has a section. Put $Z = X - U$.
Since $Z$ is a proper closed substack of $X$ and since 
$Z {\times}_X Y \to Z$ is a Nisnevich cover, we conclude from the
noetherian induction that the map $Z {\times}_X Y \to Z$ is split.
Hence the map $f$ is split.
\end{proof}
The following is our main result comparing the Nisnevich and Nisnevich 
cd-sites on stacks.
\begin{thm}\label{thm:Nis-cd-main}
Let $S$ be a noetherian and affine scheme and let $X \in {\sD}{\sM}_S$.
Then a family $\{X_i \xrightarrow{f_i} X\}$ is a covering in the Nisnevich
topology if and only if it is covering in the Nisnevich cd-topology.
In particular, the Grothendieck sites $t_{Nis}$ and 
${\left({\sD}{\sM}_S\right)}_{Nis}$ are equivalent on ${\sD}{\sM}_S$.
The Grothendieck sites $t^r_{Nis}$ and ${({\sD}{\sM}_S)}^r_{Nis}$ are
also equivalent.
\end{thm}
\begin{proof}
We prove the equivalence of the sites $t_{Nis}$ and 
${\left({\sD}{\sM}_S\right)}_{Nis}$. The equivalence between 
$t^r_{Nis}$ and ${({\sD}{\sM}_S)}^r_{Nis}$ follows exactly in the same
without any change once we note that the splitting for a representable
Nisnevich cover is given by the representable maps as is clear from the
proof of Proposition~\ref{prop:section-M}.

We have already seen that a covering in the Grothendieck topology given
by the Nisnevich $cd$-structure is also a Nisnevich covering.
To prove the converse, let $\{X_i \xrightarrow{f_i} X\}$ be a Nisnevich 
covering. Since our stacks are noetherian, we can assume that this a finite
covering. In view of Theorem~\ref{thm:split-M}, it suffices now
to show that an \'etale map $f: \wt{X} \to X$ which is split, is a covering in 
the Nisnevich $cd$-topology. We prove by the induction on the length of the 
splitting sequence. We shall construct a distinguished square of the
form ~\eqref{eqn:Dist-square} based on $X$ such that the pull-back of $f$ to
$Y$ has a section and the pull-back of $f$ to $A$ has a splitting sequence
of length less than the length of the splitting sequence for $f$.
The result will then follow by induction.

Let the map $f$ have the splitting sequence of the form 
~\eqref{eqn:s-sequence}. We take $A = X - Z_n$. To define $Y$, consider the
section $s$ of $Z_n {\times}_X \wt{X} \to Z_n$ which exists by the definition 
of the splitting sequence. Since $f$ is \'etale, it is \'etale over $Z_n$ and 
hence $s$ maps isomorphically onto an open substack $\wt{Z_n}$ of 
$Z_n {\times}_X \wt{X}$. Let $W$ be its complement. Then $W$ is a closed 
substack of $\wt{X}$ and we take $Y = \wt{X} - W$. It is immediate from the
construction that the pull-back square defined by $\{A \to X, Y \to X\}$  
is a distinguished square for the Nisnevich $cd$-structure. Moreover, the 
pull-back of $f$ to $Y$ has a section and the pull-back of $f$ to $A$ has a 
splitting sequence of length less than $n$.
\end{proof}
\begin{cor}\label{cor:Nis-cd-main**}
Let $S$ be the spectrum of a field. Then the Grothendieck sites
$\wt{t}^r_{Nis}$ and ${\left(\wt{{\sD}{\sM}}\right)}^r_{Nis}$ on
$\wt{{\sD}{\sM}}_S$ are equivalent. In other words, the site
${\left(\wt{{\sD}{\sM}}\right)}^r_{Nis}$ is given by a $cd$-structure.
\end{cor}
\begin{proof}
This follows directly from Proposition~\ref{prop:mod-sch} and
Theorem~\ref{thm:Nis-cd-main}.
\end{proof} 
\section{Some properties of Nisnevich $cd$-structure}
\label{section:Nis-CRB}
In order to prove the Nisnevich descent for the $K$-theory of perfect
complexes on stacks, we show in this section that the Nisnevich $cd$-structure
on ${\sD}{\sM}_S$ is complete, regular and bounded in the sense of 
\cite{Voev1}. We broadly follow the idea of \cite{Voev2} where this was
shown for schemes.  
We need the following results about the \'etale and unramified morphisms of 
stacks which are well known for schemes. Recall that a morphism $f : X
\to Y$ of stacks is \'etale (resp. unramified) if there are atlases
$V \xrightarrow{v} X$ and $U \xrightarrow{u} Y$ and a diagram 
\begin{equation}\label{eqn:atlas}
\xymatrix@C.8pc{
V \ar[r]^{g} \ar[d]_{v} & U \ar[d]^{u} \\
X \ar[r]_{f} & Y}
\end{equation}
such that $g$ is an \'etale (resp. unramified) map of $S$-schemes.

\begin{lem}\label{lem:Unramified}
A morphism $f : X \to Y$ in ${\sD}{\sM}_S$ is unramified if and only if the 
diagonal morphism ${\Delta}_X : X \to X {\times}_Y X$ is \'etale. 
If $f$ is representable, then it is unramified if and only if
$\Delta_X$ is an open immersion. The map $f$  
is \'etale if and only if it is smooth and unramified. 
\end{lem}
\begin{proof}
{\sl Cf.} \cite[Appendix~B.1, B.2]{DR1}.
\end{proof}
\begin{prop}\label{prop:CB}
Let $S$ be a noetherian affine scheme. The Nisnevich $cd$-structure on
${\sD}{\sM}_S$ is complete and regular. The same holds for the representable
Nisnevich $cd$-structure on ${\sD}{\sM}_S$. 
\end{prop}
\begin{proof}
Since it is clear from the definition of a Nisnevich distinguished square
that the pull-back of any distinguished square of the form 
~\eqref{eqn:Dist-square} is also a distinguished square, it follows from
\cite[Lemma~2.5]{Voev1} that the Nisnevich $cd$-structure on 
${\sD}{\sM}_S$ is complete.

We now show that the Nisnevich $cd$-structure on ${\sD}{\sM}_S$ is regular.
Let us consider a Nisnevich distinguished square of the form 
~\eqref{eqn:Dist-square} based on a stack $X$. By \cite[Lemma~2.11]{Voev1},
it suffices to show that this is a pull-back square, $e$ is a monomorphism and
the square 
\begin{equation}\label{eqn:CB1}
\xymatrix@C.8pc{
B \ar[r]^{e'} \ar[d]_{\Delta_B} & Y \ar[d]^{\Delta_Y} \\
B \times_A B \ar[r] & Y \times_X Y}
\end{equation}
is a Nisnevich distinguished square.

The first and the second assertions are obvious from the definitions
since a Nisnevich distinguished square is a pull-back square and an
open immersion of stacks is a monomorphism. So we only need to show
~\eqref{eqn:CB1}.

Note first that as $e$ is an open immersion, so is $e'$ and hence 
~\eqref{eqn:CB1} is a pull-back square. For the same
reason, $B \times_A B \to Y \times_X Y$ is an open immersion.
Let $Z$ be the complement of $A$ in $X$ and let $W = Z \times_X Y$
with reduced structures. Then $W$ is the complement of $B$ in $Y$ 
and we have $W \xrightarrow{\cong} Z$ via $p$.
Let $Z'$ be the complement of $B \times_A B$ in $Y \times_X Y$.
Then we have $Z' = \left(Y \times_X W\right) \bigcup
\left(W \times_X Y\right)$. Since the map $W \xrightarrow{p} Z$ is an
isomorphism, we see that $Z' \cong W  \times_X W \cong W \times_Z W$.
Since $W = {\Delta_Y}^{-1}(W \times_Z W) \to W \times_Z W$ is an     
isomorphism, we see that $\Delta^{-1}_Y (Z') \to Z'$ is an isomorphism.
Furthermore, as $p$ is \'etale, it follows from Lemma~\ref{lem:Unramified}
that it is unramified and hence $\Delta_Y$ is \'etale. This shows that
~\eqref{eqn:CB1} is a Nisnevich distinguished square.
The same proof also works for the completeness and boundedness of the
representable Nisnevich $cd$-structure. We only need to observe that
the diagonal maps in ~\eqref{eqn:CB1} are representable. But this
follows from \cite[Lemme~7.7]{LMB}. 
\end{proof}
\subsection{Boundedness of the Nisnevich $cd$-structures}
\label{subsection:bounded}
To show that the Nisnevich $cd$-structure on ${\sD}{\sM}_S$ is bounded,
we first recall the topological space associated to a stack from
\cite[Chapter~5]{LMB}.
Let $X$ be a stack ${\sD}{\sM}_S$. Let $|X|$ denote the set of the
equivalence classes of points $(x, K)$ on $X$. One defines the
{\sl Zariski} topology on $|X|$ by declaring a subset $U' \subset |X|$ to
be open if and only if it is of the form $|U|$, where $U$ is an open
substack of $X$. In particular, there is one-to-one correspondence between
the open (resp. reduced and closed) substacks of $X$ and open (resp. closed) 
subsets of $|X|$. Moreover, this correspondence preserves the inclusion
(and strict inclusion) relation between open (resp. reduced and closed) 
substacks of $X$ ({\sl cf.} \cite[Corollaire~5.6.3]{LMB}). It also follows 
from [{\sl loc. cit.}, Corollaire~5.7.2] that $|X|$ is a finite union of
its uniquely defined irreducible components and every irreducible closed 
subset 
$T \subset |X|$ has a unique generic point $t$ such that $T = \ov{\{t\}}$
is of the form $|Z|$ where $Z \inj X$ is an irreducible and reduced closed
substack with generic point $t$. It is also easy to check that $|X|$ is a 
noetherian topological space of dimension equal to the dimension of the 
stack $X$. 

\begin{defn}\label{defn:density}
We define a {\sl density structure} on ${\sD}{\sM}_S$ by assigning to any 
$X \in {\sD}{\sM}_S$, a sequence $D_0(X), D_1(X), \cdots$ of 
families of morphisms in ${\sD}{\sM}_S$ as follows. \\
For $i \ge 0$, let $D_i(X)$ be the family of inclusion of open substacks
$U \inj X$ such that for every irreducible and reduced closed 
substack $Z \inj (X - U)$, there is a chain of inclusions of irreducible and
reduced closed substacks $Z = Z_0 \subsetneq Z_1 \subsetneq \cdots \subsetneq
Z_i$ in $X$. A density structure such as above will be often denoted by
$D_*(-)$. Note here that all the maps in the definition of $D_*(-)$ are
closed and open immersions and hence representable.
\end{defn}
It is easy to check that the above indeed defines a density structure on
the category ${\sD}{\sM}_S$ in the sense of \cite[Definition~2.20]{Voev1}.
We define a density structure on the category ${\sN}{\sT}op$ of noetherian
topological spaces to be the standard one 
({\sl cf.} \cite[Example~2.23]{Voev1}). Thus, $D_i(X)$ is the family of
open subsets $U \subset X$ such that for every point $x \in (X-U)$,
there is sequence of points $x = x_0, x_1, \cdots , x_i$ in $X$ such that
$x_j \neq x_{j+1}$ and $x_j \in \ov{\{x_{j+1}\}}$ for each $j$. Such a 
sequence is called an increasing sequence of length $i$. 
For $(U \inj X) \in D_i(X)$, we shall often write $U \in D_i(X)$.
The following lemma is now elementary whose proof follows easily from the 
above, noting that the correspondence between the reduced and irreducible
closed substacks of $X$ and the irreducible closed subsets of $|X|$ 
preserves the strict inclusion. We leave the detail as an easy exercise.
\begin{lem}\label{lem:density1}
Let $X$ be a stack and let $U \inj X$ be an open substack. Then
$U \in D_i(X)$ if and only if $|U| \in D_i(|X|)$.
\end{lem}
$\hfill \square$
\begin{lem}\label{lem:fiber-dim}
Let $f : X \to Y$ be an \'etale morphism of Deligne-Mumford stacks.
Then the fibers of $f$ and $\ov{f} : |X| \to |Y|$ are zero-dimensional.
\end{lem}
\begin{proof}
This is also elementary and known to experts. We only give a sketch.
Let ${\rm Spec}(K) \xrightarrow{y} Y$ be a point and let $W =
{\rm Spec}(K) \times_Y X$. Then $\wt{f} : W \to {\rm Spec}(K)$ is also
\'etale. Let $U \xrightarrow{u} W$ be an \'etale atlas. Then the composite
$g \circ u :U \to {\rm Spec}(K)$ is an \'etale map of schemes and hence
$U$ is zero-dimensional and so is $W$. 

Now, let $y$ be the point of $|Y|$ defined by the point 
${\rm Spec}(K) \xrightarrow{y} Y$ above and let $T = (\ov{f})^{-1}(y)$.
Then there are natural maps $|W| \to T \to \{y\}$ such that the first map
is surjective by \cite[Proposition~5.4]{LMB}. Since we have shown that
$|W|$ is zero-dimensional, we see that $T$ must also be so.
\end{proof}
\begin{lem}\label{lem:density2}
Let $f : X \to Y$ be an \'etale and surjective morphism of stacks in 
${\sD}{\sM}_S$ and let $U \in D_i(X)$. Then $U$ defines an unique open
substack $V \inj Y$ which is in $D_i(Y)$.
\end{lem}
\begin{proof}
Since $f$ is \'etale, it is an open morphism. In particular, the image
of $U$ in $Y$ defines a unique open substack $V \inj Y$. We have to show
that $V \in D_i(Y)$. By Lemma~\ref{lem:density1}, it suffices to show that
$|V| \in D_i(|Y|)$. 

Let $y_0 \in (|Y| - |V|)$. Since $f$ is surjective, the map $\ov{f} : |X| \to
|Y|$ is clearly surjective. In particular, there is a point $x_0 \in |X|$
with $\ov{f}(x_0) = y_0$. As $x_0 \notin |U|$ and since $|U| \in D_i(|X|)$
by Lemma~\ref{lem:density1}, there is an increasing sequence
$x_0, x_1, \cdots , x_i$ of length $i$ in $|X|$. Put $y_j = \ov{f}(x_j)$.
We claim that $y_0, y_1, \cdots , y_i$ is an increasing sequence of length
$i$ in $|Y|$. 

To show that $y_j \in \ov{\{y_{j+1}\}}$, let $V'$ be an open subset of $|Y|$
containing $y_j$ and let $U' = (\ov{f})^{-1}(V')$. Then $U'$ is an open subset
of $|X|$ containing $x_j$ and hence also contains $x_{j+1}$, which in turn
implies that $y_{j+1} \in V'$, that is, $y_j \in \ov{\{y_{j+1}\}}$.
Finally, suppose $y_j = y_{j+1} = y$, say, and put $T = (\ov{f})^{-1}(y)$.
Let $T_j$ be the closure of $\{x_j\}$ in $T$. Then, the fact that $x_j \in
\ov{\{x_{j+1}\}}$ implies that $W_j \subset W_{j+1}$. Moreover, since
$W_j$ is the closure of a point in $T$, it is easy to see that it is
irreducible, that is, we get a chain $W_j \subset W_{j+1}$ of irreducible
closed subsets in $T$ which means that the dimension of $T$ is at least
one. But this is a contradiction from Lemma~\ref{lem:fiber-dim}, since
$f$ is \'etale. This completes the proof of the lemma.
\end{proof}   
\begin{lem}\label{lem:square*}
Let 
\begin{equation}\label{eqn:square1}
\xymatrix@C.8pc{
B \ar[r]^{e_Y} \ar[d]_{q} & Y \ar[d]^{p} \\
A \ar[r]_{e} & X}
\end{equation}
be a commutative square in ${\sD}{\sM}_S$ such that $e$ is an open immersion. 
Then this square is Cartesian if and only if the
square of the associated topological spaces 
\begin{equation}\label{eqn:square2}
\xymatrix@C.8pc{
|B| \ar[r]^{\ov{e_Y}} \ar[d]_{\ov{q}} & |Y| \ar[d]^{\ov{p}} \\
|A| \ar[r]_{\ov{e}} & |X|}
\end{equation}
is Cartesian.
\end{lem}
\begin{proof}
We only need to prove the `only if' part. The other one follows easily from
this in the same way. So suppose the square ~\eqref{eqn:square1} is Cartesian
and let $T = |Y| \times_{|X|} |A|$. This gives a natural map
$|B| \xrightarrow{b} T$ which factors ${\ov{e_Y}}$ and $\ov{q}$ in the
square ~\eqref{eqn:square2}. Moreover, the map $b$ is surjective by
\cite[Proposition~5.4]{LMB}. Since $e$ is an open immersion, we see that
the maps $e_Y, \ov{e}, {\ov{e_Y}}$ and $T \to |Y|$ are 
all open immersions. Thus we get the maps $|B| \to T \to |Y|$ such that
the second map and the composite map are both open immersions and the first
map is surjective. The first map must then be an isomorphism.
\end{proof} 
\begin{prop}\label{prop:bounded-main}
The Nisnevich and the representable Nisnevich $cd$-structures
on ${\sD}{\sM}_S$ are bounded with respect to
the density structure $D_*(-)$ defined above.
\end{prop} 
\begin{proof}
We show the boundedness of the Nisnevich $cd$-structure as the same
proof also works verbatim for the representable Nisnevich $cd$-structure.
Following the proof of \cite{Voev2} in the scheme case, we show that any
Nisnevich distinguished square of stacks is reducing with respect to our
density structure. So let 
\begin{equation}\label{eqn:BM1}
\xymatrix@C.8pc{
B \ar[r]^{e_Y} \ar[d] & Y \ar[d]^{p} \\
A \ar[r]_{e} & X}
\end{equation}
be a Nisnevich distinguished square. Let $B_0 \in D_{i-1}(B), \ A_0 \in
D_i(A)$ and $Y_0 \in D_i(Y)$ be given. Then we have to find another
Nisnevich distinguished square 
\begin{equation}\label{eqn:BM2}
\xymatrix@C.8pc{
B' \ar[r]^{e'_Y} \ar[d] & Y' \ar[d]^{p'} \\
A' \ar[r]_{e'} & X'}
\end{equation}
and a morphism of distinguished squares 
\begin{equation}\label{eqn:BM3}
\xymatrix@C.8pc{  
 & B' \ar[rr] \ar[dd] \ar[dl] & & Y' \ar[dd] \ar[dl]\\
B \ar[rr] \ar[dd] & & Y \ar[dd] & \\
& A' \ar[rr] \ar[dl] & & X' \ar[dl] \\
A \ar[rr] & & X &}
\end{equation}
such that $B' \to B$ factors through $B' \to B_0$, $Y' \to Y$ factors through
$Y' \to Y_0$, $A' \to A$ factors through $A' \to A_0$ and $X' \in D_i(X)$.

Since the map $(e \coprod p) : A \coprod Y \to X$ is \'etale and surjective,
and since $(A_0 \coprod Y_0) \in D_i(A \coprod Y)$, we see from 
Lemma~\ref{lem:density2} that $X_0 = (e \coprod p)(A_0 \coprod Y_0) \in
D_i(X)$. Since $D_i(X)$ is closed under finite intersections, by base changing 
to $X_0 \inj X$, we can assume that $A_0 = A$ and $Y_0 = Y$. 

We take $B' = B_0$ and $A' = A$ and let $Z = (X - A)$ and $T = B - B_0$.
Following the notations of 
Lemma~\ref{lem:square*}, we put $V' = 
|Y| - \left(\ov{\ov{e_Y}(|B-B_0|)}\right)$ and 
$W' = |X| - \left(|X - A| \bigcap \ov{(\ov{p} \circ \ov{e_Y})(|B-B_0|)}\right)$.
Let $Z_Y$ and $Z_X$ be the unique reduced closed substacks of $Y$ and $X$ 
respectively such that $|Z_Y| = \ov{\ov{e_Y}(|B-B_0|)}$ and
$|Z_X| = \ov{(\ov{p} \circ \ov{e_Y})(|B-B_0|)}$.  Then we have
$W' = |X - (Z \cap Z_X)|$ and $V' = |Y - Z_Y|$.
Put $X' = X - (Z \cap Z_X)$ and $Y' = Y - Z_Y$.
We then see that $A' = A \inj X'$ and the map $p$ then induces a map 
$Y' \to X'$, which we write as $p'$. Moreover, as $Y'$ is an open substack
of $Y$, the map $p'$ is also \'etale.

By Lemma~\ref{lem:square*}, the square ~\eqref{eqn:square2} is Cartesian,
$\ov{e}$ is an open immersion and the map $\ov{p} : (\ov{p})^{-1}(|Z|)
\to |Z|$ is an isomorphism. From this, one easily checks that the 
square
\begin{equation}\label{eqn:BM4}
\xymatrix@C.8pc{
|B'| \ar[r]^{\ov{e'_Y}} \ar[d]_{\ov{q}} & |Y'| \ar[d]^{\ov{p'}} \\
|A'| \ar[r]_{\ov{e}} & |X'|}
\end{equation}
is Cartesian, $\ov{e}$ is an open immersion and $(\ov{p})^{-1}(|X'-A'|)
\to |X'-A'|$ is an isomorphism and $|X'| \in D_i(|X|)$
({\sl cf.} \cite[Proposition~2.10]{Voev2}). 
Thus, we have a commutative diagram
\begin{equation}\label{eqn:BM5}
\xymatrix@C.8pc{
B' \ar[r]^{e'_Y} \ar[d] & Y' \ar[d]^{p'} \\
A' \ar[r]_{e'} & X',}
\end{equation}
where $e' = e$ and $e'_Y$ is the restriction of $e_Y$ on the open substack
$B'$ such that ~\eqref{eqn:BM4} is the diagram of the associated topological
spaces. Since this square is Cartesian and since $e'$ is an open immersion,
it follows from Lemma~\ref{lem:square*} that ~\eqref{eqn:BM5} is a Cartesian
square and the map $p'^{-1}(X'-A') \to (X'-A')$ is an isomorphism of the
associated reduced closed substacks. In particular, ~\eqref{eqn:BM5} is a 
Nisnevich distinguished square which clearly maps to ~\eqref{eqn:BM1}.
To complete the proof of the proposition, we only need to show that
$X' \in D_i(X)$. But this follows immediately from 
Lemma~\ref{lem:density1} since $|X'| \in D_i(|X|)$ as shown above.
\end{proof} 

\section{Nisnevich descent and other consequences}
\label{section:Consequences}
In this section, we prove the Nisnevich descent for the $K$-theory of
perfect complexes on stacks using the results of the previous sections
and we also draw some other important consequences for the cohomology
of sheaves on stacks.

Recall that $\left({\sD}{\sM}_S\right)_{Nis}$ is a Grothendieck site on the
category ${\sD}{\sM}_S$ which is noetherian. Let $*$ be a chosen final
object in the category of sets. For $X \in {\sD}{\sM}_S$,
let $\Z_X$ denote the sheaf on the restricted Grothendieck site
${{\sD}{\sM}_S}/X$ of all objects lying over $X$, which is the sheaf given
by the free abelian group on the presheaf which takes every object to
$*$. For a sheaf $\sF$ on
$\left({\sD}{\sM}_S\right)_{Nis}$, let $H^*_{Nis}(X, \sF)$ denote the 
cohomology groups $Ext^*_{Nis}(\Z_X, \sF)$.

\begin{cor}\label{cor:sheaf}
Let $\sF$ be a presheaf on $\left({\sD}{\sM}_S\right)_{Nis}$. Then $\sF$ is a
sheaf if and only if $\sF(\emptyset) = *$ and for every Nisnevich 
distinguished square of the form ~\eqref{eqn:Dist-square}, the square
\begin{equation}\label{eqn:sheaf1}
\xymatrix@C.8pc{
\sF(X) \ar[r] \ar[d] & \sF(A) \ar[d] \\
\sF(Y) \ar[r] & \sF(B)}
\end{equation}
is a pull-back.
\end{cor}
\begin{proof} 
It follows immediately from Theorem~\ref{thm:Nis-cd-main}, 
Propositions~\ref{prop:CB}, ~\ref{prop:bounded-main} and
\cite[Lemma~2.9, Corollary~2.17]{Voev1}.
\end{proof}
\begin{cor}\label{cor:cohomology}
Let $X$ be a stack of dimension $n$ and let $\sF$ be a sheaf of abelian 
groups on ${{\sD}{\sM}_S}/X$. Then
\[
H^i_{Nis}(X, \sF) = 0 \ {\rm for} \ \ i > n.
\]
\end{cor}
\begin{proof} It follows immediately from Theorem~\ref{thm:Nis-cd-main}, 
Propositions~\ref{prop:CB}, ~\ref{prop:bounded-main} and
\cite[Theorem~2.27]{Voev1}.
\end{proof}
\begin{remk}\label{remk:Zdim}
We remark that the same proof also shows that the 
Zariski cohomological dimension of a stack is bounded by its Krull dimension.
\end{remk}
%The same result also holds for the Zariski topology which is defined by
% the distinguished squares when both $p$ and $e$ are open immersions.
In the special case when $X$ is a tame stack and $\sF$ is a quasi-coherent
sheaf, the above result was proven in \cite[Theorem~1.10]{Krishna}.

\subsection{Nisnevich descent}\label{subsection:NID}
Let $\sC$ be a Grothendieck site and let $Pres({\sC})$ denote the category
of presheaves of spectra on $\sC$. Recall from \cite[Section~3]{CHSW}
that a morphism $f:\sE \to \sE'$ of presheaves of spectra is called a 
{\sl global weak equivalence} if $\sE(X) \to \sE'(X)$ is a weak equivalence
of spectra for every object $X$ in $\sC$. It is called a {\sl local weak
equivalence} if it induces an isomorphism on the sheaves of stable
homotopy groups of the presheaves of spectra. Recall from \cite{Jardine1}
(see also \cite{Jardine2}) that there is model structure, 
called the {\sl injective model structure}
on $Pres({\sC})$ for which the weak equivalence is the local weak
equivalence. Moreover, a morphism $f:\sE \to \sE'$ of presheaves of spectra 
is a cofibration if each $\sE(X) \to \sE'(X)$ is a cofibration, and
the fibrations are defined by the right lifting property with respect to
trivial cofibrations. Recall also that in the above model structure,
a {\sl fibrant replacement} of $\sE$ is a trivial cofibration $\sE \to \sE'$
such that $\sE'$ is fibrant. In particular, $\sE \to \sE'$ is a local
weak equivalence. We also recall here the following from \cite{CHSW}.
\begin{defn}\label{defn:desc-def}
A presheaf of spectra $\sE$ on $\sC$ is said to satisfy the {\sl descent}
in the Grothendieck topology of the site $\sC$ if the fibrant replacement
map $\sE \to \sE'$ is a global weak equivalence. 
\end{defn}
For the presheaf of non-connective spectra $\sK$ on ${\sD}{\sM}_S$ which 
associates to a stack $X$, the $K$-theory spectrum of perfect complexes on 
$X$, let $\sK^{Nis}$ denote its fibrant replacement in the Grothendieck 
topology $({\sD}{\sM}_S)_{Nis}$. One says that $K$-theory satisfies the 
Nisnevich descent if the presheaf of spectra $\sK$ satisfies descent in
the Nisnevich topology ${\left({\sD}{\sM}\right)}_{Nis}$ on stacks.
For any $q \in \Z$, let $\sK_q$ denote the sheaf of $q$th stable homotopy 
groups of $\sK$ on ${\sD}{\sM}_S$. Thus, $\sK_q$ is the sheafification of
the presheaf $X \mapsto K_q(X)$.
The following is our main result about the Nisnevich descent.
\begin{thm}\label{thm:MNDS}
Let $S$ be the spectrum of a field $k$ and let $\wt{{\sT}{\sD}{\sM}}_S$ 
denote the full subcategory of $\wt{{\sD}{\sM}}_S$ consisting of tame stacks. 
Then, the $K$-theory of perfect complexes satisfies the Nisnevich descent on 
the Grothendieck site $\wt{{\sT}{\sD}{\sM}}^r_S$. In particular,
if $char(k) = 0$, then $K$-theory satisfies the Nisnevich descent on
$\wt{{\sD}{\sM}}^r_S$.
\end{thm}
\begin{proof} It is an elementary fact that the fiber product of
representable maps of tame stacks is also tame.
%Uses Lagrange's theorem to show that if $X \to Y$ is a representable
%map of DM stacks and $Y$ is tame, then so is $X$.
The theorem now follows immediately from Corollaries~\ref{cor:MV},
~\ref{cor:Nis-cd-main**}, Propositions~\ref{prop:CB}, 
~\ref{prop:bounded-main} and \cite[Theorem~3.4]{CHSW}.
\end{proof}
\begin{cor}\label{cor:MNDS*}
Let $S$ be the spectrum of a field and let $X$ be in $\wt{{\sT}{\sD}{\sM}}_S$.
Then there is a strongly convergent spectral sequence 
\[
E^{p,q}_2 = H^p_{t^r_{Nis}}\left(X, \sK_q\right) \Rightarrow
K_{q-p}(X).
\]
\end{cor}
\begin{proof} This is an immediate consequence of Theorem~\ref{thm:MNDS}
and Corollary~\ref{cor:cohomology}.
\end{proof}
The following is also an immediate consequence of Theorem~\ref{thm:MNDS}. 
\begin{cor}\label{cor:MNDS**} 
Let $S$ be the spectrum of a field and let $X$ be in $\wt{{\sT}{\sD}{\sM}}_S$.
Then for any representable Nisnevich cover $\sU \to X$, the natural
map
\[
K_*(X) \to \H^*(\sU, \sK)
\]
is an isomorphism.
\end{cor}
$\hfil \square$
\begin{remk}\label{remk:G-maps*}
If $S = {\rm Spec}(k)$ with $char(k) = 0$ and $G$ is an affine smooth
group scheme over $S$, then it follows from the Remark~\ref{remk:G-maps}
and Theorem~\ref{thm:MNDS} that the equivariant $K$-theory satisfies the
Nisnevich descent on the category ${Sch}^G_S$ of schemes of finite type
over $S$ with proper action of $G$.
\end{remk}
\begin{exm}\label{exm:Wrong-Nisnevich}
Let $S$ be a noetherian scheme and let $G$ be a smooth and affine group
scheme over $S$. Let ${Sch}^G_S$ denote the category of schemes of finite type
over $S$ with proper action of $G$. We have shown in Remark~\ref{remk:G-maps}
that there is a naturally defined Nisnevich site on ${Sch}^G_S$ where
a $G$-equivariant map $f : X \to Y$ is a {\sl Nisnevich cover} if
the associated map of quotient stacks $[X/G] \to [Y/G]$ is a Nisnevich
cover in the category ${\sD}{\sM}_S$. We have also shown that this site
on ${Sch}^G_S$ is given by a complete, regular and bounded $cd$-structure.
In \cite{Serpe1}, an {\sl isovariant Nisnevich} site on ${Sch}^G_S$ is studied,
where a $G$-equivariant map $f : X \to Y$ is defined to be a Nisnevich cover
if $f$ is a Nisnevich cover of schemes, forgetting the $G$-action and it
preserves the isotropy groups.
We show in the following example that this site is not a reasonable site to 
consider and may not be given by a $cd$-structure. 

So let $k = \R$ and let $G = {\Z}/2$. We set $X = {\rm Spec}(\C)$ and
$Y = {\rm Spec}(\R)$ with the obvious map $f : X \to Y$. Note that there
is a $G$-action on $X$ induced by the conjugation on $\C$ and $f$ is 
the quotient map. We consider the commutative 
diagram
\begin{equation}\label{eqn:exm2}
\xymatrix{
G \times X \ar[r]^{\phi} \ar[d]_{p} & X \ar[d]^{f} \\
X \ar[r]_{f} & Y,}
\end{equation}
where $\phi$ is the action map which is $G$-equivariant if we let $G$ act
diagonally on $G \times X$ for the trivial action on $X$ and the left 
multiplication action on $G$. The map $p$ is the projection.
Since $f$ is a principal bundle quotient, 
the above diagram is Cartesian. Since $G$ acts freely on itself, it acts freely
on $G \times X \cong X \coprod X$, and the map $\phi$ takes each component
isomorphically onto $X$ as schemes. 
Thus, we find that $\phi$ is a $G$-equivariant
map which is Nisnevich and isovariant. However, the map $f$ of quotients
is not a Nisnevich map. This shows in particular that the isovariant
Nisnevich site ({\sl cf.} \cite{Serpe1}) of $X$ is not equivalent to
the usual Nisnevich site of $X/G$. This also shows that the isovariant
Nisnevich site on ${Sch}^G_S$ may not be given by a $cd$-structure.
\end{exm} 
\noindent\emph{Acknowledgments.} 
The first author would like to thank A. Hogadi for various discussions 
which helped him in understanding some basic facts about stacks.

\end{document}